\documentclass[12pt]{elsarticle}

\usepackage[english]{babel}
\usepackage{amsthm,amsmath,amssymb}
\usepackage{graphicx}
\usepackage[small,bf,hang]{caption} 

\usepackage{subfig}

\usepackage{algorithm}                   
\usepackage{algorithmic}

\usepackage{epsf} 

\usepackage{hyperref}
\usepackage{enumerate} 

\usepackage{multirow}

\usepackage{booktabs} 

\usepackage[margin=1in]{geometry}

\usepackage{color}

\journal{???}

\newtheorem{theorem}{Theorem}[section]
\theoremstyle{plain}

\newtheorem{lemma}[theorem]{Lemma}

\newtheorem{observation}[theorem]{Observation}

\theoremstyle{definition}

\graphicspath{{figures/}}

\begin{document}

\begin{frontmatter}

\title{Graphs with few Hamiltonian Cycles}

\author[gent,mons]{Jan Goedgebeur\fnref{fwo}}
\ead{jan.goedgebeur@ugent.be}

\author[gent]{Barbara Meersman}
\ead{barbara.meersman@ugent.be}

\author[gent,cluj]{Carol T.\ Zamfirescu\fnref{fwo}}
\ead{czamfirescu@gmail.com}

\address[gent]{Department of Applied Mathematics, Computer Science \& Statistics\\
  Ghent University\\
Krijgslaan 281-S9 \\9000 Ghent, Belgium\medskip}

\address[mons]{Computer Science Department\\
  University of Mons\\
Place du Parc 20 \\7000 Mons, Belgium\medskip}

\address[cluj]{Department of Mathematics\\
  Babe\c{s}-Bolyai University\\
Cluj-Napoca, Roumania\\ }


\fntext[fwo]{Supported by a Postdoctoral Fellowship of the Research Foundation Flanders (FWO).}

%
%




\begin{abstract}


We describe an algorithm for the exhaustive generation of non-isomorphic graphs with a given number $k \ge 0$ of hamiltonian cycles, which is especially efficient for small $k$. Our main findings, combining applications of this algorithm and existing algorithms with new theoretical results, revolve around graphs containing exactly one hamiltonian cycle (1H) or exactly three hamiltonian cycles (3H). Motivated by a classic result of Smith and recent work of Royle, we show that there exist nearly cubic 1H graphs of order $n$ iff $n \ge 18$ is even. This gives the strongest form of a theorem of Entringer and Swart, and sheds light on a question of Fleischner originally settled by Seamone. We prove equivalent formulations of the conjecture of Bondy and Jackson that every planar 1H graph contains two vertices of degree~2, verify it up to order~16, and show that its toric analogue does not hold. We treat Thomassen's conjecture that every hamiltonian graph of minimum degree at least $3$ contains an edge such that both its removal and its contraction yield hamiltonian graphs. We also verify up to order~21 the conjecture of Sheehan that there is no 4-regular 1H graph. Extending work of Schwenk, we describe all orders for which cubic 3H triangle-free graphs exist. We verify up to order~$48$ Cantoni's conjecture that every planar cubic 3H graph contains a triangle, and show that there exist infinitely many planar cyclically 4-edge-connected cubic graphs with exactly four hamiltonian cycles, thereby answering a question of Chia and Thomassen. Finally, complementing work of Sheehan on 1H graphs of maximum size, we determine the maximum size of graphs containing exactly one hamiltonian path and give, for every order $n$, the exact number of such graphs on $n$ vertices and of maximum size.

\end{abstract}


\end{frontmatter}


\section{Introduction}
\label{section:intro}

In 1946, Smith showed that every edge in a cubic graph is contained in an even number of hamiltonian cycles; see Tutte's paper~\cite{Tu46} or Berge's book~\cite[pp.~189--190]{berge1973graphs}. Thus, a hamiltonian cubic graph contains at least three hamiltonian cycles, so among cubic graphs there exist no graphs with exactly one hamiltonian cycle, i.e.\ \emph{uniquely hamiltonian} graphs. Strengthening Smith's result, Thomason proved in 1978 that in a graph containing only vertices of odd degree, every edge is contained in an even number of hamiltonian cycles~\cite{Th78}. Thus, uniquely hamiltonian graphs without vertices of even degree, and in particular $k$-regular uniquely hamiltonian graphs do not exist for odd $k$. What about even $k$? Using Lov\'asz' Local Lemma, Thomassen~\cite{thomassen1998independent} proved that $k$-regular uniquely hamiltonian graphs do not exist for even $k \ge 300$, and with a careful choice of parameters his theorems give 73 instead of 300. This was improved by Haxell, Seamone, and Verstraete~\cite{haxell2007independent} to $k \ge 23$. Sheehan conjectured that there are no 4-regular uniquely hamiltonian graphs~\cite{sheehan1975multiplicity}. By Petersen's 2-Factor Theorem, the truth of this conjecture would imply that cycles are the only regular uniquely hamiltonian graphs.

In another article~\cite{sheehan1977graphs}, Sheehan studied the maximum size of uniquely hamiltonian graphs and proved that such a graph on $n$ vertices contains at most $\lfloor \frac{n^2}{4} \rfloor + 1$ edges. He provides for each $n \ge 3$ an $n$-vertex uniquely hamiltonian graph of maximum size and states that these are the only uniquely hamiltonian graphs of this size. Barefoot and Entringer~\cite{barefoot1981census} proved that Sheehan erred for $n \ge 9$ by showing that for every $n \ge 7$ there exist exactly $2^{\lceil \frac{n}{2} \rceil -4}$ uniquely hamiltonian graphs of maximum size.


Thomason's aforementioned result implies that a uniquely hamiltonian graph must have at least two vertices of even degree. This relationship between a graph's degrees and whether or not it is uniquely hamiltonian raises some natural questions, for instance whether there are uniquely hamiltonian graphs of minimum degree~3. Entringer and Swart answered this question affirmatively by describing an infinite family of \emph{nearly cubic} graphs, i.e.\ graphs with exactly two vertices of degree~4 and all other vertices cubic~\cite{entringer1980spanning}. Fleischner~\cite{fleischner2014uniquely} recently showed that there exist uniquely hamiltonian graphs in which every vertex has degree~4 or~14.  

Bondy and Jackson~\cite{bondy1998vertices} proved that a uniquely hamiltonian graph of order~$n$ has at least one vertex of degree at most $c \log_2 8n + 3$, i.e.\ the minimum degree cannot be greater than this number, where $c \approx 2.41$. Abbasi and Jamshed~\cite{abbasi2006degree} improved this to $c \log_2 n + 2$, where $c \approx 1.71$. In their article, Bondy and Jackson were particularly interested in \emph{planar} uniquely hamiltonian graphs. They showed that such a graph must contain at least two vertices of degree 2 or 3, and conjectured that every planar uniquely hamiltonian graph contains at least two vertices of degree~2.

This paper is structured as follows. In Section~\ref{section:generation_algo} we describe our algorithm for the generation of graphs with ``few'' hamiltonian cycles---we emphasise that this includes the important class of non-hamiltonian graphs. Thereafter, in Section~\ref{sect:results}, we present new theoretical results which we combine with the conclusions derived from our implementation of the generation algorithm, as well as existing algorithms. More specifically, in Section~\ref{subsect:results_uhg} we determine all orders for which uniquely hamiltonian nearly cubic graphs exist, thereby giving the strongest form of a theorem of Entringer and Swart~\cite{entringer1980spanning}. This result extends work of Royle~\cite{Ro17} and addresses a question of Fleischner~\cite{fleischner2014uniquely} originally settled by Seamone~\cite{seamone2015uniquely}. We give equivalent formulations of the conjecture of Bondy and Jackson~\cite{bondy1998vertices} mentioned above, verify it up to order~16, but also present a uniquely hamiltonian graph on the torus with exactly one 2-valent vertex. We also treat the conjecture of Thomassen~\cite{thomassen1996number} that every hamiltonian graph $G$ of minimum degree at least~$3$ contains an edge $e$ such that both $G - e$ (remove the edge but not its endpoints) and $G/e$ (contract the edge) are hamiltonian. It is elementary to see that Thomassen's conjecture holds for all graphs with at least two hamiltonian cycles, but it is open for uniquely hamiltonian graphs. We present the orders of all hamiltonian 4-regular graphs up to order~21 with the minimum number of hamiltonian cycles, extending work of Haythorpe~\cite{haythorpe2017minimum}. This verifies up to order~21 Sheehan's conjecture~\cite{sheehan1975multiplicity} that no 4-regular uniquely hamiltonian graph exists.

In Section~\ref{subsect:results_thg}, motivated by a classic result of Smith, we study cubic graphs with exactly three hamiltonian cycles. Note that these three hamiltonian cycles together cover each edge exactly twice and thus form a \emph{cycle double cover}. There exist small such graphs of girth~3 in abundance---however, we show that up to order~32 there are only two cubic triangle-free graphs containing exactly three hamiltonian cycles (two generalised Petersen graphs), but that starting from order~34 every even order is covered. This extends a result of Schwenk~\cite{schwenk1989enumeration}. Making use of a theorem of Thomason~\cite{Th78}, we prove that a graph in which every vertex has odd degree and which has exactly $p$ hamiltonian cycles, where $p$ is prime, must be 3-connected. Thus, any counterexample to Cantoni's conjecture~\cite{Tutte79} stating that every planar cubic graph with exactly three hamiltonian cycles contains a triangle, is 3-connected. We verify this conjecture up to order~48. We also show that for every $k$ that is 0 or at least 4 there exists a planar cyclically 4-edge-connected cubic graph with exactly $k$ hamiltonian cycles, while by Thomason's result mentioned in the first paragraph, no such graph exists for $k \in \{ 1,2 \}$. For many of the above conjectures we establish significantly better bounds if a lower bound on the girth is imposed.

In Section~\ref{subsect:chia_thomassen} we prove that there exist infinitely many planar cyclically 4-edge-connected cubic graphs with exactly four hamiltonian cycles, thereby answering a question of Chia and Thomassen~\cite{chia2012}.

In Section~\ref{subsect:results_utg}, naturally complementing work of Sheehan~\cite{sheehan1977graphs} on the size of uniquely hamiltonian graphs, we give structural results on graphs containing exactly one hamiltonian path and determine their maximum size using a result of Barefoot and Entringer~\cite{barefoot1981census}. We also give the exact number of such graphs of maximum size. The paper ends with Section~\ref{sect:comput_results} in which we briefly comment on the implementation of our algorithm and on correctness testing. The total computational effort for this project amounted to 40 CPU years.

\section{Generation of graphs with few hamiltonian cycles}
\label{section:generation_algo}

We describe an algorithm to generate all pairwise non-isomorphic graphs of a given order~$n$, containing, for a fixed non-negative integer $k$, exactly $k$ hamiltonian cycles. More specifically, in Section~\ref{subsect:gen_uhg} we present an algorithm for generating uniquely hamiltonian graphs, while in Section~\ref{subsect:gen_extensions} we describe how this algorithm can be extended to generate graphs with $k \neq 1$ hamiltonian cycles efficiently. Our experiments indicate that our algorithm is significantly more efficient than previous algorithms for $k \leq 7$. 
To the best of our knowledge, all previously available methods to generate exhaustively graphs with exactly $k$ hamiltonian cycles consisted in using a program such as \textit{geng}~\cite{nauty-website, mckay_14} to generate all graphs with a given order, to then use a separate program to count the number of hamiltonian cycles of the generated graphs, and finally to filter the graphs with the desired number of hamiltonian cycles. 

In Section~\ref{sect:results} we present the computational results which we obtained with our implementation of this algorithm, together with new theoretical results. In Section~\ref{sect:comput_results} we report the running times of the algorithm and how we tested the correctness of our implementation.

\subsection{Generation of uniquely hamiltonian graphs}
\label{subsect:gen_uhg}

To generate all uniquely hamiltonian graphs of a given order~$n$, we start the algorithm from a cycle of order~$n$ and, in essence, recursively add edges to it in all possible ways as long as the graph stays uniquely hamiltonian. It is clear that all uniquely hamiltonian graphs of order $n$ can be obtained in this way.

To guarantee that the algorithm does not output isomorphic copies, we use McKay's canonical construction path method~\cite{mckay_98}. In order to use this approach, we first have to define a \textit{canonical reduction} which is unique up to isomorphism. An \textit{expansion} is an operation which constructs a larger graph from a given graph, while the reverse operation is called a \textit{reduction}. We call an expansion that is the inverse of a canonical reduction a \textit{canonical expansion}. The two rules of the canonical construction path method are:

\begin{enumerate}

\item Only accept a graph if it was constructed by a canonical expansion.

\item For every graph $G$ to which expansion operations are applied, only perform one expansion from each equivalence class of expansions of $G$.

\end{enumerate}

The pseudocode of our algorithm to generate all uniquely hamiltonian graphs of order $n$ can be found in Algorithm~\ref{algo:generate_uhg} (recall that we start the algorithm from a cycle of order $n$). We will now explain how we applied the canonical construction path method for the generation of uniquely hamiltonian graphs. In Theorem~\ref{thm:proof_complete} we then prove that our algorithm indeed generates all pairwise non-isomorphic uniquely hamiltonian graphs of a given order $n$.

\begin{algorithm}[ht!]
\caption{Construct(graph $G$)}
  \begin{algorithmic}[1]
  \label{algo:generate_uhg}
  \STATE Output $G$
  \STATE Determine a list $L$ of pairs of non-adjacent vertices in $G$
  \STATE Determine orbits of $L$
  \FOR{one representative vertex-pair $\{a,b\}$ in every orbit of $L$}
  	\STATE Add edge $ab$ to $G$
  	\IF{this expansion was canonical and $G$ is still uniquely hamiltonian} \label{line:uh_test}
		\STATE Construct($G$)
	\ENDIF  	
	\STATE Remove edge $ab$ from $G$
  \ENDFOR
  \end{algorithmic}
\end{algorithm}

In our case, there is only one expansion operation: to insert an edge between two non-adjacent vertices. We implement the second rule of the canonical construction path method by first computing the orbits of all pairs of non-adjacent vertices and only applying the expansion operation for one representative pair from each orbit. (We use the program \textit{nauty}~\cite{nauty-website, mckay_14} to determine all generators of the automorphism group of a graph. Thereafter, we use a union-find algorithm to determine the orbits of pairs of non-adjacent vertices.)

A \textit{reducible} edge is an edge $e$ in $G$ for which $G-e$ is still uniquely hamiltonian. So every edge of $G$ is reducible, except for the $n$ edges of the unique hamiltonian cycle of $G$. For the first rule of the canonical construction path method we first have to define a canonical reduction which is unique up to isomorphism. In order to do so efficiently, we assign a 9-tuple $(x_0,\ldots,x_8)$ to every reducible edge of a uniquely hamiltonian graph $G$ and define a \textit{canonical edge} as a reducible edge with the lexicographically maximal value for this 9-tuple. The canonical reduction is defined as the reduction of a canonical edge.

We denote by ${\mathfrak h}$ the hamiltonian cycle of $G$. For a reducible edge $e = ab$, the invariants $x_0,\ldots,x_6$ are invariants of increasing discriminating power and cost and are defined as follows:

\begin{itemize}	
\addtolength{\itemsep}{-1mm}
\item $x_0$ ($x_1$) is the maximum (minimum) of the degrees of $a$ and $b$.
\item $x_2$ is the negative of the minimum length between $a$ and $b$ on ${\mathfrak h}$.
\item $x_3$ ($x_4$) is the negative of the maximum (minimum) of the sum of the degrees of the vertex preceding and succeeding $a$ on ${\mathfrak h}$ and the sum of the degrees of the vertex preceding and succeeding $b$ on ${\mathfrak h}$.
\item $x_5$ is the number of common neighbours of $a$ and $b$.
\item $x_6$ is the negative of the number of vertices at distance at most 2 of $a$ and $b$.
\end{itemize}



The values $x_0,\ldots,x_6$ are invariant under isomorphisms, but in principle two non-equivalent edges can have the same value for $(x_0,\ldots,x_6)$. Therefore we define $\{x_7,x_8\}$ as the lexicographically largest label of an edge which is in the same orbit as $e$ in the canonical labelling of the graph. (We use the program \textit{nauty}~\cite{nauty-website, mckay_14} to compute a canonical labelling.)

For the correctness of the algorithm it would be sufficient only to compute the values of $x_7$ and $x_8$, but as computing a canonical labelling is computationally expensive, it is much more efficient to use the other invariants $x_i$, $ 0 \leq i \leq 6$, as well.

More specifically, we first compute the value of $x_0$ and $x_1$ for every reducible edge of the graph. Since we require a canonical edge to have maximal value for $(x_0,\ldots,x_8)$, we only need to compute $x_{i+1}$ for the reducible edges which have maximal value for $(x_0,\ldots,x_i)$. Furthermore, if the edge $e$ which was added by the last expansion operation is no longer in the list of reducible edges with maximal value for $(x_0,\ldots,x_i)$, we do not have to compute $x_{i+1}$ as we already know that our last expansion was not canonical. Similarly, if $e$ is the only edge with maximal value for $(x_0,\ldots,x_i)$, we do not have to compute $x_{i+1}$ as we already know that our last expansion was canonical.

The discriminating power of the invariants $x_0,\ldots,x_6$ is usually sufficient to avoid the more expensive computation of $x_7$ and $x_8$. For example, this is so in 99.4\,\% of the cases when generating uniquely hamiltonian graphs of order 11. Finally, note that the invariants $x_0$ and $x_1$ allow a look-ahead: in many cases it is easy to determine upfront that inserting an edge between two non-adjacent vertices $a$ and $b$ cannot be canonical as there will be other edges with a larger degree vector.

The proof of Theorem~\ref{thm:proof_complete} is analogous to the corresponding arguments in the proofs of other generation algorithms---nevertheless, we have chosen to include it for completeness' sake.

\begin{theorem} \label{thm:proof_complete}
When Algorithm~\ref{algo:generate_uhg} is applied to an $n$-cycle, it outputs exactly one representative of every isomorphism class of uniquely hamiltonian graphs of order~$n$.
\end{theorem}
\begin{proof}
First we prove that at least one representative of every isomorphism class of uniquely hamiltonian graphs of order $n$ is generated. Assume by induction on the number of edges $m$ that every uniquely hamiltonian graph of order $n$ with at most $m$ edges is generated and accepted by the algorithm. Consider a uniquely hamiltonian graph $G$ of order $n$ with $m+1$ edges. Since $m+1 > m \geq n$, $G$ contains a reducible edge, thus also a canonical edge $ab$. By definition $G-ab$ is uniquely hamiltonian, so by induction a graph $H$ isomorphic to $G-ab$ was generated. Let $\gamma$ be an isomorphism from $G-ab$ to $H$. The graph $H$ has a pair of non-adjacent vertices $\{c,d\}$ which are in the same orbit of non-adjacent vertices as $\{\gamma(a),\gamma(b)\}$ to which the edge insertion operation is applied. This produces a graph $H+cd$ which is isomorphic to $G$ and let $\gamma'$ be an isomorphism from $G$ to $H+cd$.
The edge $cd$ is in the same orbit of edges as $\gamma'(ab)$ under the action of the automorphism group of $H+cd$. This implies that $cd$ has maximal value for $(x_0,\ldots,x_8)$, so $H+cd$ is accepted by the algorithm.

Now we show that at most one representative of every isomorphism class of uniquely hamiltonian graphs of order $n$ is generated. Assume by induction on the number of edges $m$ that every uniquely hamiltonian graph of order $n$ with at most $m$ edges is generated at most once by the algorithm. Let $G_1$ and $G_2$ be two isomorphic uniquely hamiltonian graphs of order $n$ with $m+1$ edges that are both accepted by the algorithm. Let $\gamma$ be an isomorphism from $G_1$ to $G_2$ and let $e_i=a_ib_i$ be the canonical edge from $G_i$ which was added in the last step of the algorithm, for $i \in \{1,2\}$ (so $G_i$ was obtained from $G_i-e_i$ by adding $e_i$).
Since $e_1$ and $e_2$ are both canonical edges, $\gamma(e_1)$ is in the same orbit of edges as $e_2$ under the action of the automorphism group of $G_2$. So there is an automorphism of $G_2$ which maps $\gamma(e_1)$ to $e_2$. But this automorphism induces an isomorphism $\gamma'$ from $G_1-e_1$ to $G_2-e_2$.
Thus, by our induction hypothesis, $G_1-e_1$ and $G_2-e_2$ are the same graph and hence $\gamma'$ is an automorphism which maps $\{a_1, b_1\}$ to $\{a_2,b_2\}$.
So $\{ a_1, b_1\}$ is in the same orbit of non-adjacent vertex pairs as $\{ a_2, b_2\}$, while our algorithm only inserts an edge for one representative of every orbit of non-adjacent vertex pairs. 
\end{proof}

\subsection{Extensions of the generation algorithm}
\label{subsect:gen_extensions}

\subsubsection{Generation of graphs with $k > 1$ hamiltonian cycles}
\label{subsect:gen_k_hc}

Our algorithm for uniquely hamiltonian graphs from Section~\ref{subsect:gen_extensions} can be easily extended to generate (hamiltonian) graphs with at most or exactly $k > 1$ hamiltonian cycles. The modified algorithm still starts from a cycle of order $n$, but now on line~\ref{line:uh_test} of Algorithm~\ref{algo:generate_uhg} we have to test if $G$ has at most $k$ hamiltonian cycles instead of testing if $G$ is uniquely hamiltonian. Note that some of the invariants in $x_0,\ldots,x_6$ have to be adapted or omitted as e.g.\ $x_2$ relies on the fact that the graph only contains one hamiltonian cycle. In particular the invariants $x_2$, $x_3$ and $x_4$ are omitted. However, the discriminating power of the remaining invariants $x_0,x_1,x_5,x_6$ is still sufficient to avoid the more expensive computation of $x_7$ and $x_8$ in most cases. (For example: when generating graphs of order 11 with exactly three hamiltonian cycles, this can be avoided in 98.6\,\% of the cases.)

If we only want to generate hamiltonian graphs with exactly (instead of at most) $k$ hamiltonian cycles, we nevertheless have to count the number of hamiltonian cycles before outputting the graphs. (Note that this method is not particularly efficient to generate graphs with exactly $k$ hamiltonian graphs for large values of $k$.)
It is also clear that our algorithm is not very efficient to generate all graphs with $k$ hamiltonian cycles for large values of $k$---however, our experiments indicate that our algorithm is significantly more efficient than previous algorithms for $k \leq 7$.

The only other modification which is required is in the definition of reducible edge. Given a hamiltonian graph $G$ with at most $k$ hamiltonian cycles, an edge $e$ of $G$ is \emph{reducible} if and only if $G-e$ is hamiltonian.

The proof of the following theorem is analogous to the proof of Theorem~\ref{thm:proof_complete} and therefore omitted.

\begin{theorem}
When the modified version of Algorithm~\ref{algo:generate_uhg} is applied to a cycle of order $n$, the algorithm outputs exactly one representative of every isomorphism class of hamiltonian graphs of order $n$ with at most $k$ hamiltonian cycles.
\end{theorem}
%

\subsubsection{Generation of non-hamiltonian graphs}

As in Section~\ref{subsect:gen_k_hc}, Algorithm~\ref{algo:generate_uhg} can be adapted to generate non-hamiltonian graphs of order $n$ efficiently. To this end, one has to start the algorithm from a graph consisting of $n$ isolated vertices (instead of a cycle of order $n$). Now every edge is reducible when performing the generation. One can also adjust this for a specialised algorithm for only generating connected non-hamiltonian graphs by starting the generation from all trees on $n$ vertices and defining an edge of a connected non-hamiltonian graph $G$ to be reducible if and only if $G-e$ is connected. However, as most non-hamiltonian graphs are connected, this will not be much faster than generating all non-hamiltonian graphs and filtering the connected graphs. (For example: more than 90\,\% of the non-hamiltonian graphs on 11 vertices are connected, and with increasing order this ratio increases as well.)

As in Section~\ref{subsect:gen_k_hc}, the invariants $x_2$, $x_3$ and $x_4$ are omitted as they rely on the fact that the graph is uniquely hamiltonian. However, the discriminating power of the remaining invariants is still sufficient to avoid the more expensive computation of $x_7$ and $x_8$ in most cases. (For example: when generating non-hamiltonian graphs of order 11, this can be avoided in 91.2\,\% of the cases.)


We close this section by mentioning that, since the algorithm only adds edges and never removes any edges, it is straightforward to extend it and restrict the generation to graphs with a given lower bound on the girth, planar graphs, graphs with an upper bound on the maximum degree, and various other properties.


\section{Results}
\label{sect:results}

\subsection{Uniquely hamiltonian graphs}
\label{subsect:results_uhg}

For a graph $G$, we shall denote by $h(G)$ the number of hamiltonian cycles it contains.

\subsubsection{Nearly cubic uniquely hamiltonian graphs}

Cubic uniquely hamiltonian graphs do not exist by Smith's theorem. Following Entringer and Swart~\cite{entringer1980spanning}, we call an $n$-vertex graph \emph{nearly cubic} if exactly $n - 2$ of its vertices are cubic, while the remaining two vertices are of degree~4. A uniquely hamiltonian graph contains at least two vertices of even degree, as we will show in the following lemma which essentially belongs to Thomason:

\begin{lemma}\label{Thom-lemma}
A graph $G$ with $h(G) \in \{ 1,2 \}$ contains at least $3 - h(G)$ vertices of even degree.
\end{lemma}

\begin{proof}
Thomason~\cite{Th78} showed that in a graph containing only vertices of odd degree, every edge is contained in an even number of hamiltonian cycles. Therefore, a hamiltonian graph containing only vertices of odd degree has at least three hamiltonian cycles. So in a graph with exactly one or two hamiltonian cycles there must be at least one vertex of even degree. Suppose there exists a uniquely hamiltonian graph $G$ with exactly one vertex $u$ of even degree. If the degree of $u$ is not 2 (and thus at least 4), removing from $G$ an edge $uv$ which does not lie on the hamiltonian cycle of $G$ yields a uniquely hamiltonian graph with all vertices of odd degree except for $v$. We iterate this procedure until the degree of the vertex of even degree, which we call $w$, is 2, and we denote the neighbours of $w$ by $w'$ and $w''$. Let $G_1$ and $G_2$ be disjoint copies of $G - w$, and $w'_i$ and $w''_i$ the respective copies of $w'$ and $w''$. Then $G_1 \cup G_2$ to which we add the edges $w'_1w'_2$ and $w''_1w''_2$ is a uniquely hamiltonian graph in which all vertices have odd degree. However, this contradicts Thomason's theorem mentioned in the beginning of this proof. Thus, a uniquely hamiltonian graph must contain at least two vertices of even degree.
\end{proof}

By Euler's degree sum formula, every nearly cubic graph has even order. Entringer and Swart~\cite{entringer1980spanning} showed that for all even $n \ge 22$ there exists a nearly cubic uniquely hamiltonian graph of order $n$. Recently, Royle presented a nearly cubic uniquely hamiltonian graph on 18 vertices~\cite{Ro17}, noting that this is the smallest such graph.
By modifying our algorithm from Section~\ref{section:generation_algo} for nearly cubic graphs, we verified this independently, addressed the last remaining open case, and determined the exact counts of such graphs for the smallest orders for which they occur:

\begin{theorem}\label{Th-nc}
There exists a nearly cubic uniquely hamiltonian graph of order $n$ if and only if $n$ is even and $n \ge 18$. Royle's graph has girth~$5$ and is the only nearly cubic uniquely hamiltonian graph on $18$~vertices. Furthermore, there are exactly $20$ nearly cubic uniquely hamiltonian graphs of order~$20$, $337$ of order~$22$, and $4592$ of order~$24$. Finally, both the smallest nearly cubic uniquely hamiltonian graph of girth~$3$ as well as of girth~$4$ have order~$20$.
\end{theorem}

Royle's nearly cubic uniquely hamiltonian graph on 18~vertices is shown in Figure~\ref{fig:nearly_cubic_18v_g5}. There is precisely one nearly cubic uniquely hamiltonian graph of girth~4 on 20~vertices and it is shown in Figure~\ref{fig:nearly_cubic_20v_g4}. There are exactly 17 such graphs of girth~3 on 20~vertices, one of which is shown in Figure~\ref{fig:nearly_cubic_20v_g3}.
The nearly cubic graphs up to 24 vertices can also be downloaded from the \textit{House of Graphs}~\cite{hog} at \url{http://hog.grinvin.org/UHG}.

\begin{figure}[h!tb]
    \centering
   \subfloat[]{\label{fig:nearly_cubic_18v_g5}\includegraphics[width=0.25\textwidth]{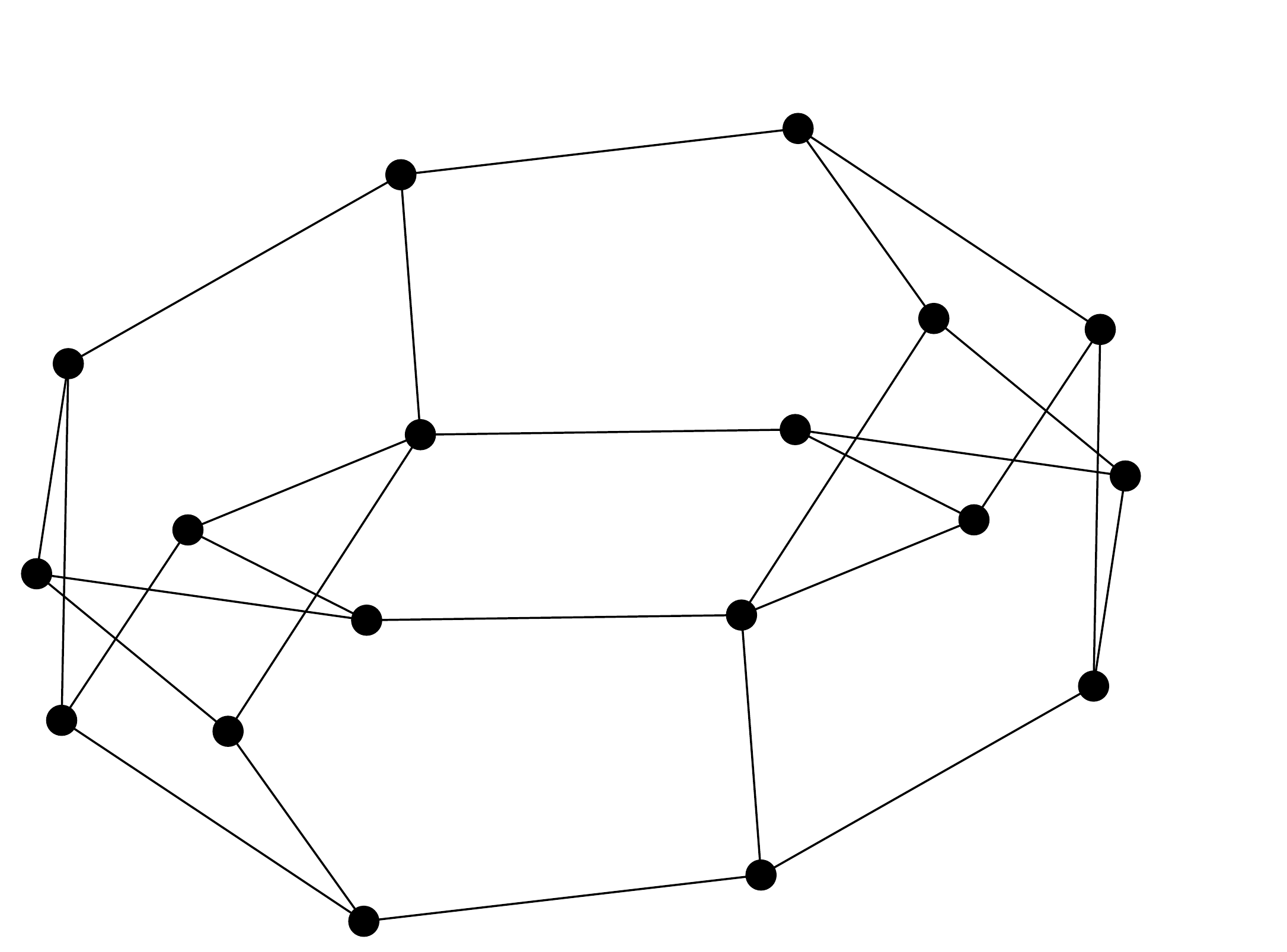}} \qquad
    \subfloat[]{\label{fig:nearly_cubic_20v_g4}\includegraphics[width=0.26\textwidth]{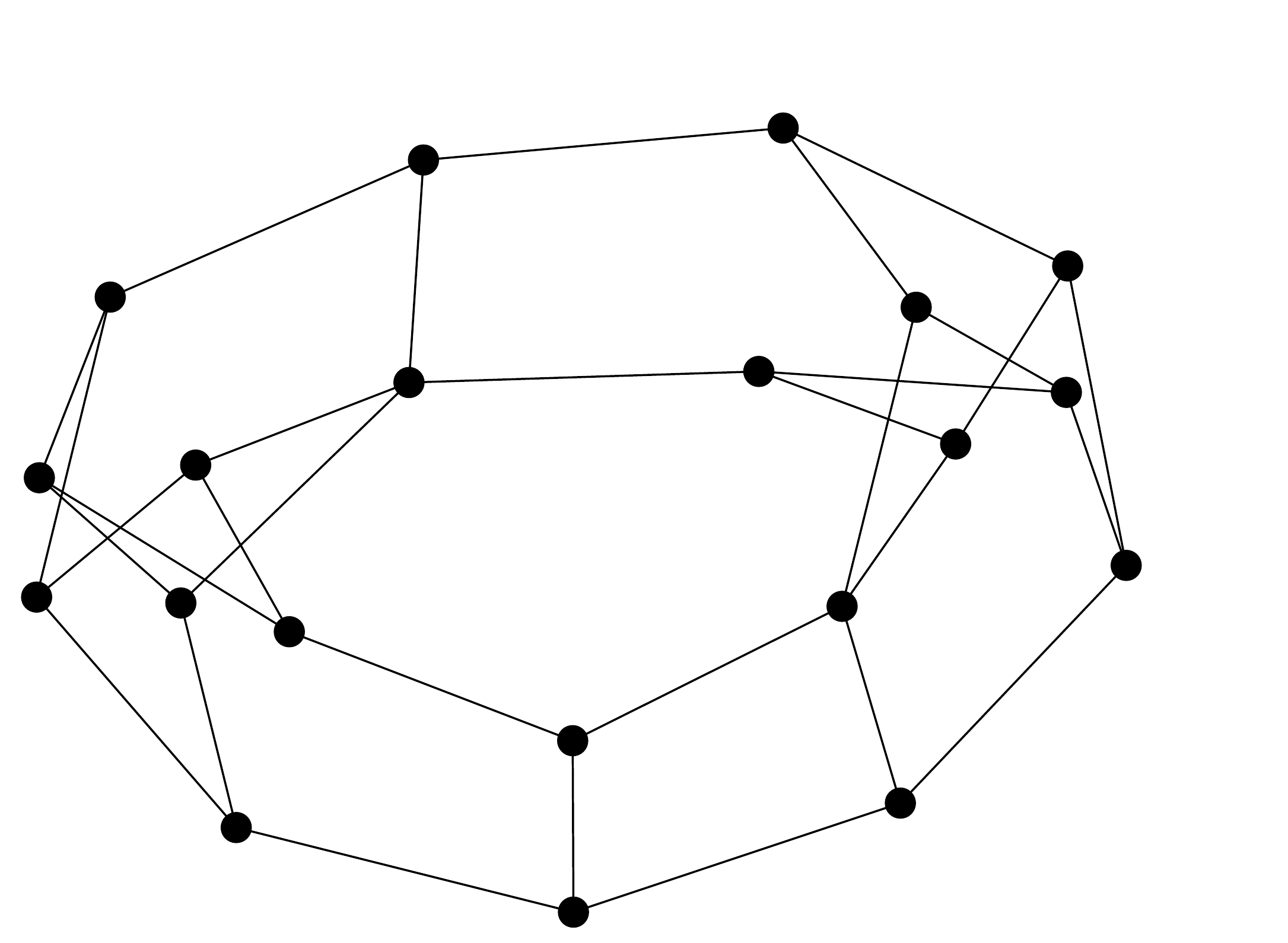}} \qquad
    \subfloat[]{\label{fig:nearly_cubic_20v_g3}\includegraphics[width=0.30\textwidth]{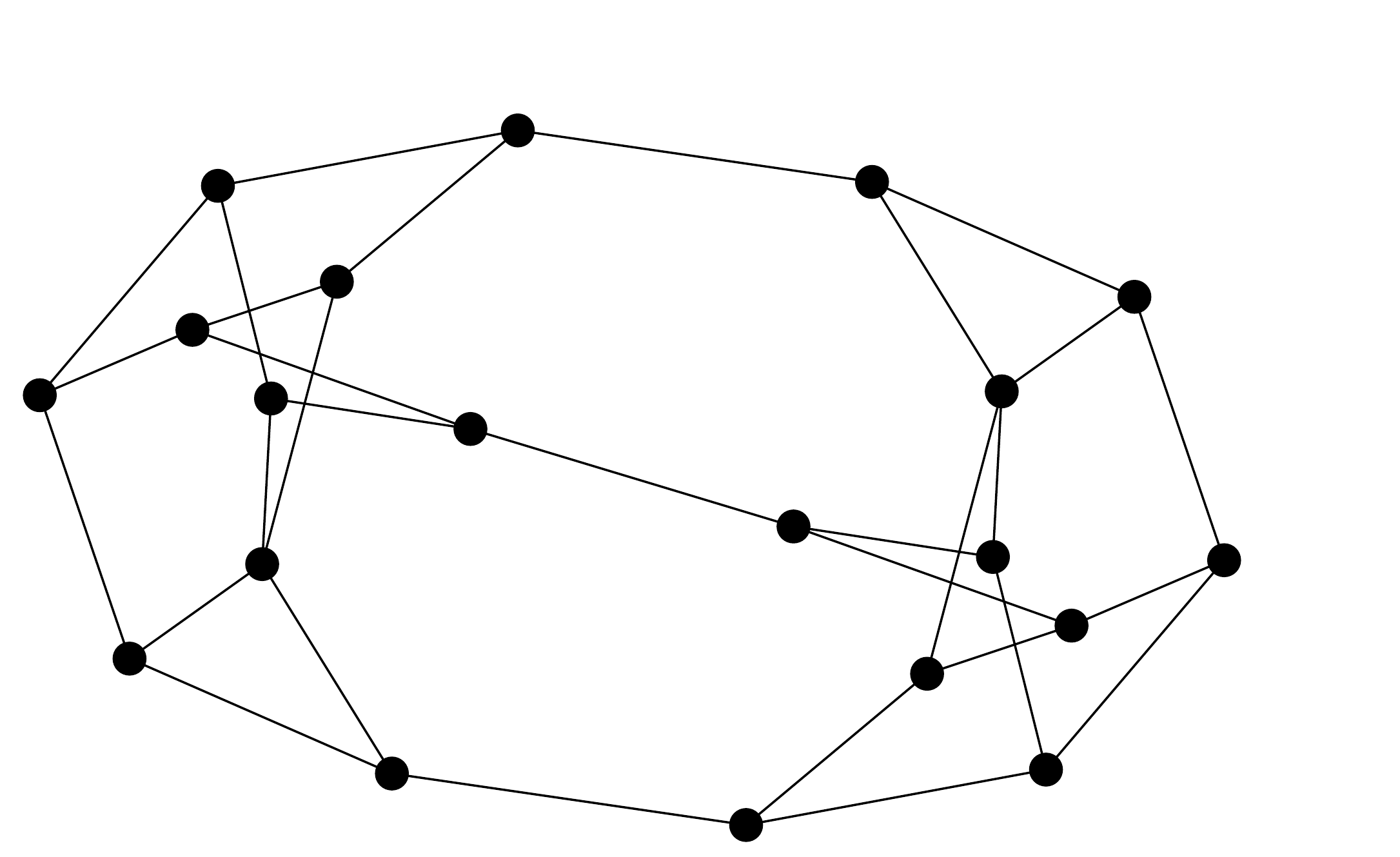}}

    \caption{The smallest nearly cubic uniquely hamiltonian graphs of (a) girth~5, due to Royle, (b)~girth~4, and (c) girth~3. Their orders are 18, 20, and 20, respectively.}
\end{figure}

Our focus on girth stems from a question of Fleischner~\cite[p.~176]{fleischner2014uniquely}, who asked whether there exist uniquely hamiltonian graphs without 2-valent vertices and of girth $>3$. Entringer and Swart's aforementioned approach~\cite{entringer1980spanning} yields graphs containing exactly two triangles. Seamone provides in~\cite{seamone2015uniquely} a method to construct nearly cubic triangle-free uniquely hamiltonian graphs, thereby giving an affirmative answer to Fleischner's question. However, he does not discuss concrete examples, in particular small ones. The graph mentioned above due to Royle~\cite{Ro17} provides such a concrete example and Royle showed that there is no smaller uniquely hamiltonian graph of minimum degree at least~3. Theorem~\ref{Th-nc} expands on this.

\subsubsection{The Bondy-Jackson conjecture}

\bigskip

\noindent \textbf{Conjecture} (Bondy and Jackson~\cite{bondy1998vertices}). \emph{Every planar uniquely hamiltonian graph contains at least two vertices of degree~$2$.} \hfill $({\mathfrak B})$

\bigskip

Using our implementation of Algorithm~\ref{algo:generate_uhg} from Section~\ref{section:generation_algo}, we generated all planar uniquely hamiltonian graphs with girth at least 3, 4, and 5 up to certain orders, see Table~\ref{table:counts_planar_uhg}. (We used Boyer and Myrvold's algorithm~\cite{boyer2004cutting} to test if a graph is planar.)

\begin{table}[h!tb]
	\centering
	\small
	\begin{tabular}{crrr}
		\toprule	
		Order &  \#\,planar UH graphs &  girth $\geq 4$ & girth $\geq 5$ \\
		\midrule
		3 & 1 & 0 & 0 \\
		4 & 2 & 1 & 0 \\
		5 & 3 & 1 & 1 \\
		6 & 12 & 2 & 1 \\
		7 & 49 & 3 & 1  \\
		8 & 460 & 11 & 3 \\
		9 & 4 994 & 33 & 4 \\
		10 & 68 234 & 178 & 8 \\
		11 & 997 486 & 1 011 & 23 \\
		12 & 15 582 567 & 6 816 & 91 \\
		13 & 253 005 521 & 47 669 & 317 \\
		14 & 4 250 680 376 & 352 901 & 1 353 \\
		15 & 73 293 572 869 & 2 680 512 & 6 473 \\
		16 & 1 293 638 724 177 & 20 939 433 & 30 834 \\
		17 & ? & 166 713 951 & 148 907 \\	
		18 & ? & 1 352 143 860 & 768 178 \\
		19 & ? & 11 129 922 982 &  3 987 517 \\
		20 & ? & ? & 20 767 030 \\
		21 & ? & ? & 110 819 167 \\
		22 & ? & ? & 599 311 836 \\
		23 & ? & ? & 3 256 610 004 \\	
		\bottomrule		
	\end{tabular}
	\caption{The number of all planar uniquely hamiltonian graphs, all such graphs with girth at least~4, and all such graphs with girth at least~5, respectively. Every graph in this table contains at least two 2-valent vertices.}
	\label{table:counts_planar_uhg}
\end{table}

While performing these computations, we verified that none of the generated graphs is a counterexample to $({\mathfrak B})$. That is:

\begin{observation}
The conjecture of Bondy and Jackson $({\mathfrak B})$ is true for graphs up to order~$16$, for graphs of girth at least~$4$ up to order~$19$, and for graphs of girth at least~$5$ up to order~$23$.
\end{observation}

If we relax the planarity condition in $({\mathfrak B})$ to ``having genus~1", we can give a counterexample on only 11~vertices---it is shown in Figure~\ref{fig:toric_UHG}. More specifically, we determined that there are exactly two uniquely hamiltonian graphs with at most one vertex of degree~2 on 11~vertices, exactly 20 on 12~vertices, none on smaller orders, and that all of these 22~examples have genus~1. Furthermore, the smallest toric counterexample of girth~4 (girth~5) has order~13 (order~14). Using the aforementioned findings and observing that we can replace cubic vertices with triangles iteratively without altering neither the number of hamiltonian cycles nor the number of 2-valent vertices, nor the genus, we obtain the following result.

\begin{observation}
There exists an $n$-vertex toric uniquely hamiltonian graph containing at most one $2$-valent vertex if and only if $n \ge 11$.
\end{observation}

\begin{figure}[!htb]
\begin{center}
 \includegraphics[width=0.3\textwidth]{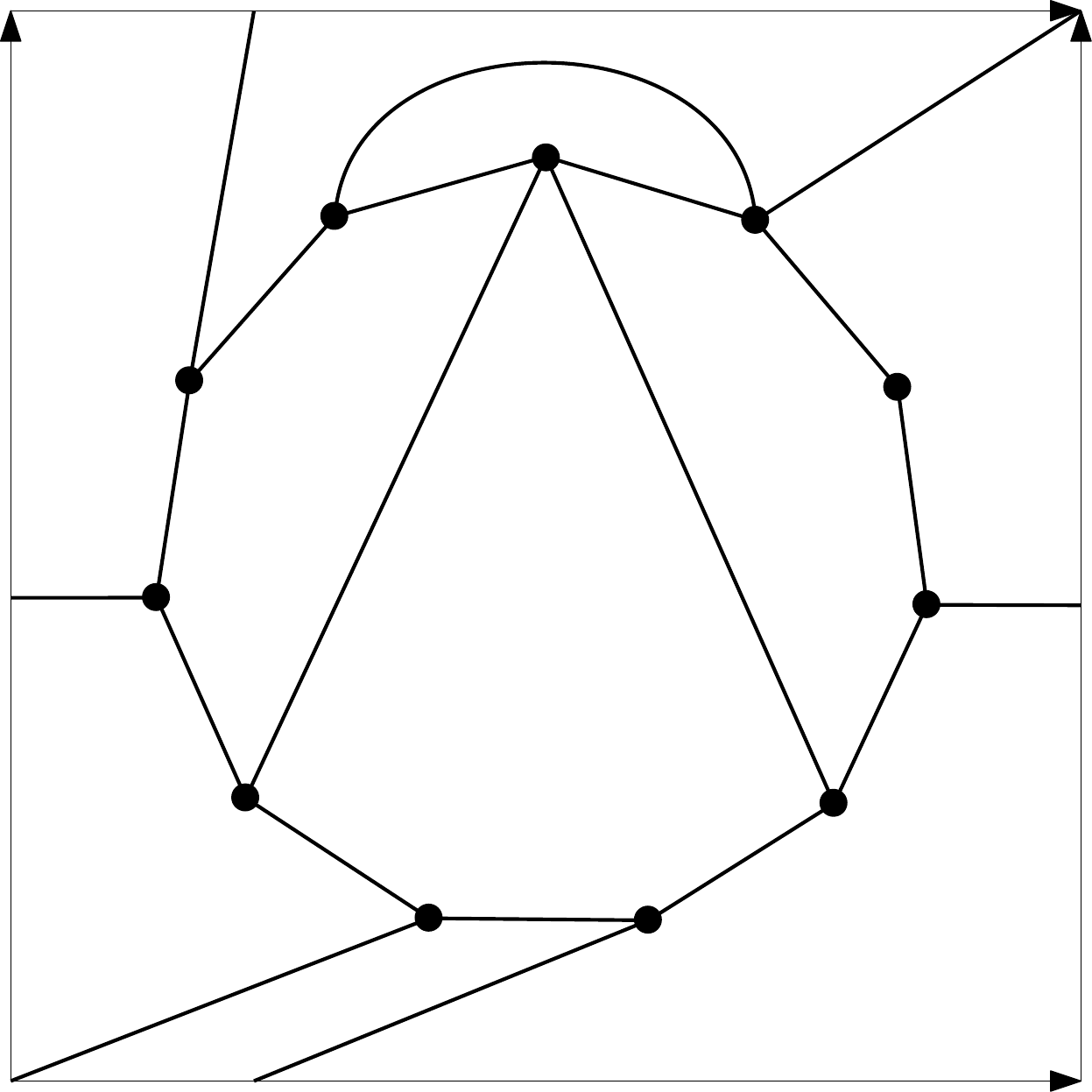}
	\caption{A toric uniquely hamiltonian graph containing exactly one vertex of degree~2. Its order is 11. There are no smaller uniquely hamiltonian graphs containing at most one 2-valent vertex, irrespective of genus.}
	\label{fig:toric_UHG}
\end{center}
\end{figure}



While we were able to find uniquely hamiltonian graphs of minimum degree~3 on the double torus---one such graph of order~18 can be constructed from the graph depicted in Figure~\ref{fig:toric_UHG}---it remains an open question to establish the existence of a toric uniquely hamiltonian graph of minimum degree at least~3.

What if we replace ``uniquely hamiltonian graph'' by ``containing exactly two hamiltonian cycles'' in the Bondy-Jackson conjecture? By subdividing an arbitrary edge once in a planar cubic graph with exactly three hamiltonian cycles---discussed in detail in Section~\ref{subsect:results_thg}---, we obtain infinitely many planar graphs with exactly two hamiltonian cycles having precisely one 2-valent vertex. The smallest such graph is $K_4$ with one subdivided edge. The next theorem addresses, among other things, the situation for minimum degree at least~3.

\begin{theorem}
The following statements are equivalent.
\begin{enumerate}
\item There exists a counterexample to $({\mathfrak B})$ of minimum degree at least~$3$.
\item There exists a counterexample to $({\mathfrak B})$ containing exactly one vertex of degree~$2$.
\item There exist infinitely many counterexamples to $({\mathfrak B})$.
\item There exists a planar graph with exactly two hamiltonian cycles and minimum degree at least~$3$.
\item There exists a planar non-hamiltonian $n$-vertex graph with exactly one $(n-1)$-cycle, no vertex of degree~$0$ or $1$ and at most one vertex of degree~$2$.
\end{enumerate}
\end{theorem}

\begin{proof}
Let $G$ be a counterexample to $({\mathfrak B})$ of minimum degree at least~3. Consider an edge $e$ on the unique hamiltonian cycle of $G$. Adding a vertex on $e$ yields a counterexample to $({\mathfrak B})$ containing exactly one 2-valent vertex. Now let $G$ be a counterexample to $({\mathfrak B})$ containing precisely one vertex of degree~2. We use the same idea as in the proof of Lemma~\ref{Thom-lemma} and obtain a counterexample to $({\mathfrak B})$ of minimum degree at least~3. We have shown the equivalence of statements 1 and 2. These arguments immediately yield that if one counterexample to $({\mathfrak B})$ exists, there must be infinitely many such counterexamples.

We now show the equivalence of the existence of a counterexample to $({\mathfrak B})$ and statement~4. Suppose $G$ is a planar graph with minimum degree at least~3 and containing exactly two hamiltonian cycles ${\mathfrak h}_1$ and ${\mathfrak h}_2$. We can then subdivide once an edge in $E({\mathfrak h}_1) \setminus E({\mathfrak h}_2)$ and obtain a uniquely hamiltonian graph $G'$. All vertices in $G$ (which we see as a subgraph of $G'$) have the same degree as in $G'$, so $G'$ contains exactly one vertex of degree~2, and clearly $G'$ is planar. This means that we have obtained a counterexample to $({\mathfrak B})$.

Assume now that there is a counterexample $G$ to $({\mathfrak B})$. Without loss of generality we may suppose that $G$ has minimum degree at least~3. We denote the hamiltonian cycle in $G$ by ${\mathfrak h}$. Consider, for adjacent vertices $x,y$ in $G$ such that $xy$ lies on ${\mathfrak h}$, the graph $$G' = (V(G) \cup \{ v_1, v_2 \}, E(G) \setminus \{ xy \} \cup \{ xv_1, xv_2, v_1v_2, v_1y, v_2y \} ).$$
There are exactly two hamiltonian $xy$-paths in $G'[\{ x, y, v_1, v_2 \}]$ and there is exactly one hamiltonian $xy$-path in $G'[V(G)]$. Thus, $G'$ is a planar graph with minimum degree~3 and containing exactly two hamiltonian cycles.

Let us prove the equivalence of the existence of a counterexample to $({\mathfrak B})$ and statement~5. Let $G$ be an $n$-vertex counterexample to $({\mathfrak B})$, which by above discussion we may assume to contain a 2-valent vertex $w$ with neighbours $w',w''$. Let $G_1$ and $G_2$ be disjoint copies of $G - w - w_1w_2$ (we remove the edge $w_1w_2$ only if it is present in $G$), and $w'_i$ and $w''_i$ the respective copies of $w'$ and $w''$. In $G_1 \cup G_2$, we identify $w'_1$ with $w'_2$ and $w''_1$ with $w''_2$, calling the respective resulting vertices $v'$ and $v''$. We then add the edge between $v'$ and $v''$ and a vertex $v$ on this new edge. We obtain the graph $G'$ of order $2n - 3$. The graph $G'$ is non-hamiltonian but contains exactly one $(2n - 4)$-cycle (which avoids $v$) and is clearly planar.

Let $G$ be a graph satisfying the properties given in statement~5. We denote its unique $(n-1)$-cycle by ${\mathfrak c}$, and by $v$ the vertex not contained in ${\mathfrak c}$. Let $v_1, \ldots, v_d$ be the neighbours of $v$. Observe that $G$ may contain a vertex of degree 2 and that this can be $v$. This makes no difference in the arguments that follow. The graph $G - v = G'$ is uniquely hamiltonian and planar. If all but at most one vertices of $G'$ have degree at least 3, then we have a counterexample to $({\mathfrak B})$ and we are done.

The degrees of the vertices which are not $v_1, \ldots, v_d$ remain unchanged when we remove $v$ from $G$. Thus, we must now take care of the degrees of $v_1, \ldots, v_d$. This is achieved by adding edges $v_i v_{i+1}$ where necessary. Note that if $G' + v_iv_{i+1}$ contains a new hamiltonian cycle ${\mathfrak h}$, then this cycle must use $v_iv_{i+1}$; but then replacing in ${\mathfrak h}$ the edge $v_iv_{i+1}$ with the path $v_ivv_{i+1}$, we obtain a hamiltonian cycle in $G$, a contradiction, as $G$ was assumed to be non-hamiltonian. If the degrees of $v_i$ and $v_{i+1}$ in $G'$ are already at least~3 then it is not necessary to add an edge. Otherwise the degrees of $v_i$ and $v_{i+1}$ are at least 2 since from each vertex only exactly one incident edge was removed. (In $G$, the vertex $v_i$ could not have had degree~2 as then ${\mathfrak c}$ would have visited $v$.)

We still need to deal with the case when $v_i v_{i+1} \in E(G)$. We now prove that in this situation the degrees of $v_i$ and $v_{i+1}$ were already sufficiently large. If $v_i$ and $v_{i+1}$ are adjacent in $G$, then $v_i v_{i+1} \notin E({\mathfrak c})$, as discussed above. Thus, each of $v_i$ and $v_{i+1}$ is incident with at least two further edges, namely the edges lying on ${\mathfrak c}$. Together with the edge $v_i v_{i+1}$ we obtain that the degrees of $v_i$ and $v_{i+1}$ are at least 3 in $G'$, and that it is not necessary to add an edge between these two vertices.

Therefore, we can modify the graph $G'$ such that all of its vertices have degree at least~3, with the possible exception of at most one vertex of degree~2 already present in $G$. It is clear that the edges $v_i v_{i+1}$ can be added such that the graph remains planar. This graph we have constructed is a counterexample to $({\mathfrak B})$.
\end{proof}

The proof of the following observation consists of two parts. Firstly, consider the infinite family of planar non-hamiltonian graphs with a unique $(n-1)$-cycle and exactly two vertices of degree~2 from Figure~\ref{fig:palm}, left-hand side. The maximum degree of this family of graphs is unbounded. The right-hand side of Figure~\ref{fig:palm} shows an infinite family of such graphs in which the maximum degree of each member is bounded above by~4. Secondly, using a computer, we verified that there exist no planar non-hamiltonian graphs of order $n < 10$ containing exactly one $(n-1)$-cycle, exactly two vertices of degree~$2$ and all other vertices of degree at least $3$. 

\begin{observation}
There exists a planar non-hamiltonian graph of order $n$ containing exactly one $(n-1)$-cycle, exactly two vertices of degree~$2$, and all other vertices of degree at least $3$ if and only if $n \ge 10$.
\end{observation}

\begin{figure}[!ht]
\begin{center}
\includegraphics[height=48mm]{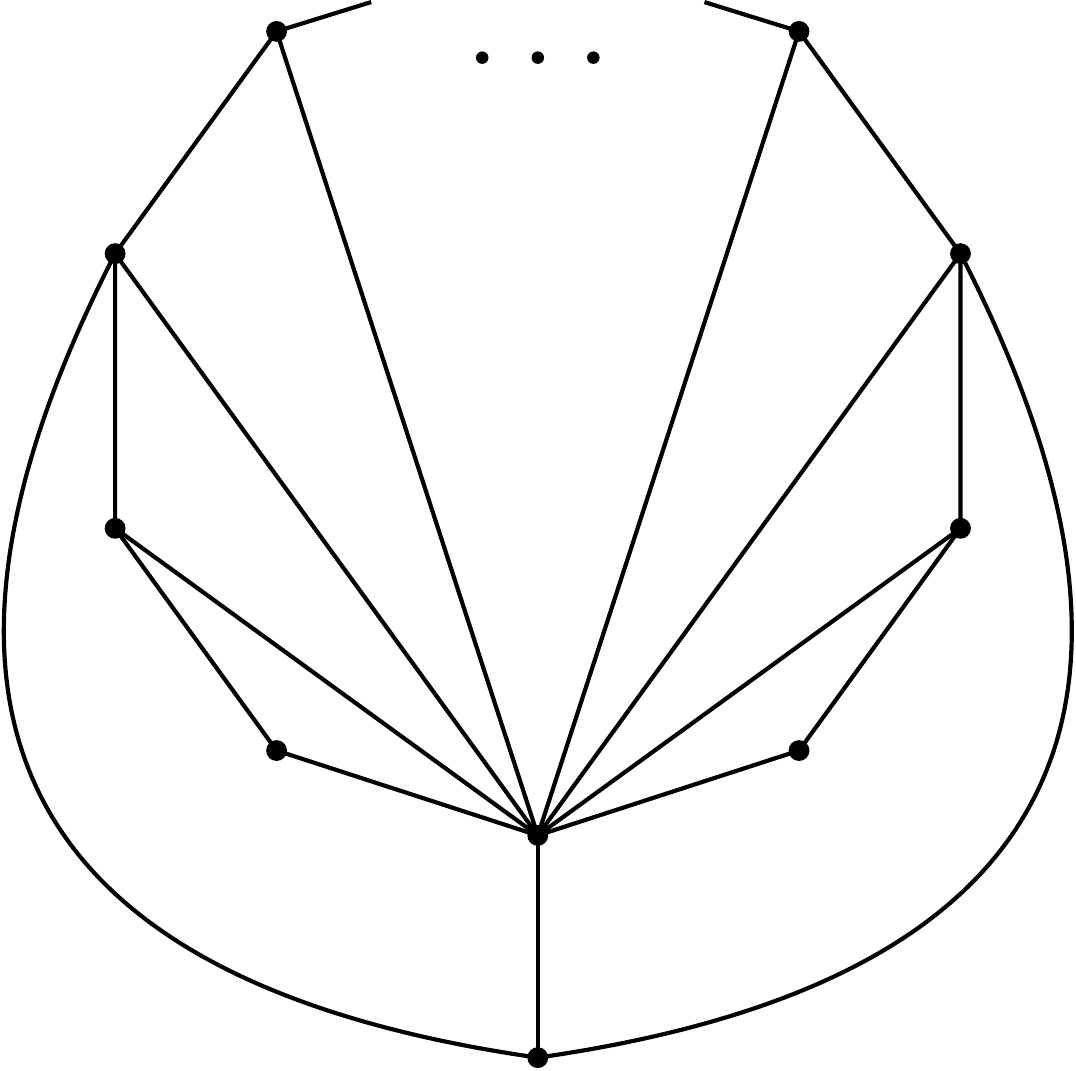} \hspace{2cm} \includegraphics[height=50mm]{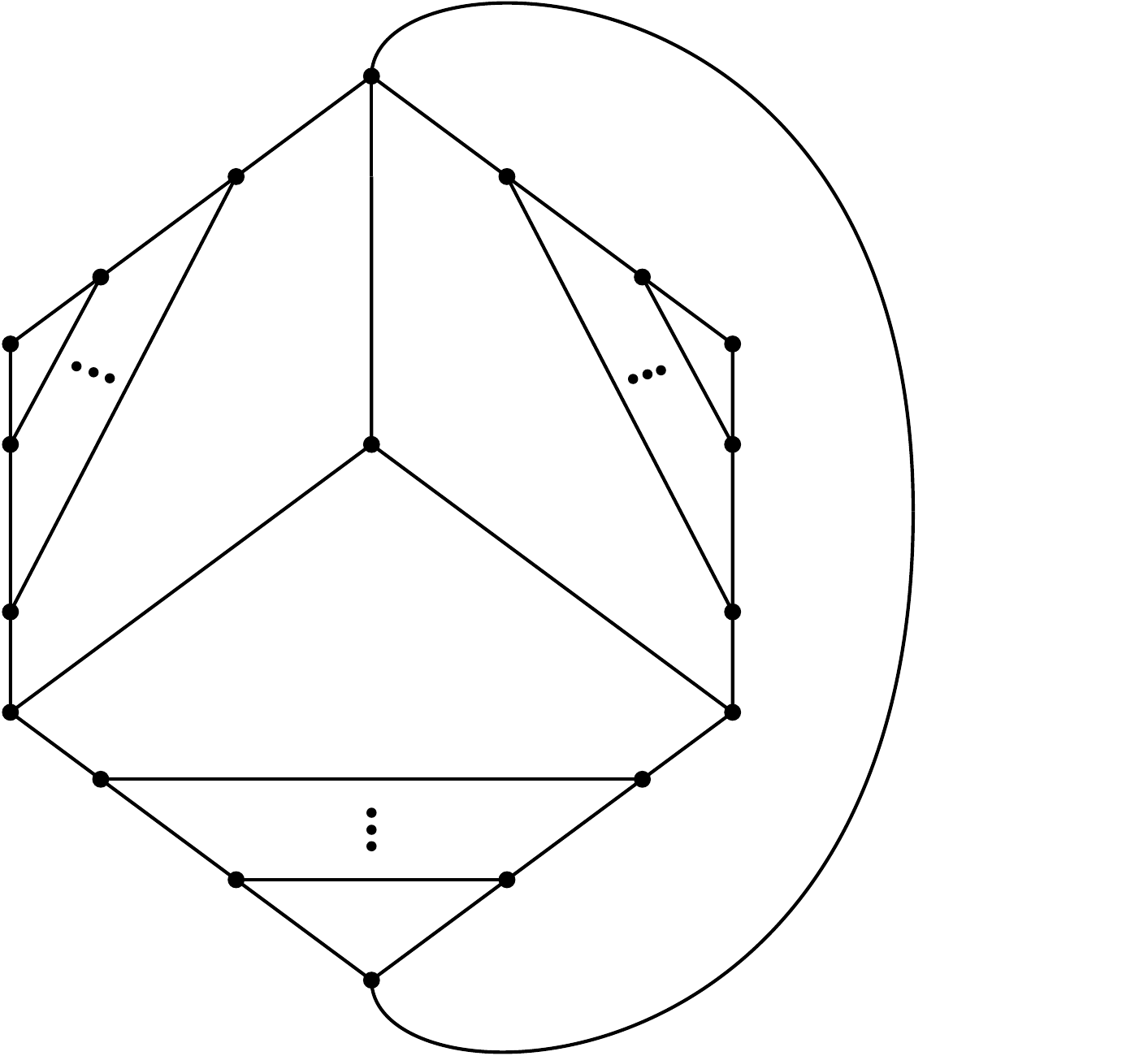}
	\caption{Two infinite families of planar non-hamiltonian graphs with a unique $(n-1)$-cycle and exactly two vertices of degree~2.}
	\label{fig:palm}
\end{center}
\end{figure}

\subsubsection{Thomassen's conjecture and a question of Royle}

It is a natural question to ask for the smallest order of a uniquely hamiltonian graph with a certain minimum degree $\delta$ or connectivity $\kappa$. For $\delta = 2$ and $\kappa = 2$ the answer is trivial, namely $K_3$. There are many such graphs as presented in other sections of this article. For both $\delta = 3$ and $\kappa = 3$ the answer is 18 and was given by Royle~\cite{Ro17} using a computer, see Figure~\ref{fig:nearly_cubic_18v_g5} for an example. In~\cite{fleischner2014uniquely}, Fleischner describes uniquely hamiltonian graphs in which every vertex has degree~4 or 14. His smallest example of connectivity~2 has 338~vertices, while the smallest graph of connectivity~3 he constructs has order~408. Thus, the smallest uniquely hamiltonian graph with $\delta = 4$ has order at least~18 and at most~338. To the best of our knowledge, no uniquely hamiltonian graph of minimum degree at least~5 is known.

Concerning $\kappa \ge 4$, Fleischner conjectured~\cite[p.~176]{fleischner2014uniquely} that every uniquely hamiltonian graph has connectivity at most~3. It seems that the first explicit construction of a 3-connected uniquely hamiltonian graph is due to Grinberg~\cite{grinberg1986three}, and his example is of the same order (18) and only one edge larger than the smallest example there is, which was determined by Royle using a computer~\cite{Ro17}. Aldred and Thomassen also described a 3-connected uniquely hamiltonian graph, see~\cite{holton1999planar}. By replacing in these graphs cubic vertices (which they all contain) with triangles, we obtain infinitely many such graphs. Using Seamone's technique~\cite{seamone2015uniquely} we can render them triangle-free.

We now state the following intriguing conjecture of Thomassen concerning uniquely hamiltonian graphs with minimum degree at least~3, and discuss its connection with a recent question of Royle.

\bigskip

\noindent \textbf{Conjecture} (Thomassen~\cite{thomassen1996number}). \emph{Every hamiltonian graph $G$ of minimum degree at least $3$ contains an edge $e$ such that both $G - e$ and $G/e$ are hamiltonian.} \hfill $({\mathfrak T})$

\bigskip

(Recall that $G - e$ stands for removing from $G$ the edge $e$ but not its endpoints and $G/e$ denotes the graph obtained by contracting the edge $e$ from $G$.)
If $G$ contains at least two hamiltonian cycles, then any edge contained in one but not the other hamiltonian cycle satisfies the condition from $({\mathfrak T})$, so the conjecture's veracity is open exactly for the family of uniquely hamiltonian graphs. Let a graph $G$ contain exactly one hamiltonian cycle ${\mathfrak h}$. Clearly, $G - e$ is non-hamiltonian for every $e \in E({\mathfrak h})$, so candidate edges $vw$ satisfying $({\mathfrak T})$ must lie in $E(G) \setminus E({\mathfrak h})$. For such an edge $vw$ we have that $G/vw$ is hamiltonian if and only if $G - v$ or $G - w$ is hamiltonian (possibly both).

Royle recently asked (personal communication) whether uniquely hamiltonian graphs of minimum degree~3 and without $(n - 1)$-cycles exist. If such a graph would be found, it would constitute a counterexample to $({\mathfrak T})$.

Recall that Royle~\cite{Ro17} showed that the smallest uniquely hamiltonian graphs with minimum degree at least~3 have order~18. Royle's result implies the first statement of the following observation. The other statements were obtained using our implementation of Algorithm~\ref{algo:generate_uhg} combined with a separate program which tests if each generated graph $G$ contains an edge~$e$ such that both $G - e$ and $G/e$ are hamiltonian.

\begin{observation}\label{obs:conj_thomassen}
Thomassen's conjecture $({\mathfrak T})$ is true for graphs up to order~$17$, for graphs of girth at least~$4$ up to order~$18$, for graphs of girth at least~$5$ up to order~$22$, and for nearly cubic graphs up to order~$24$.
\end{observation}

The number of uniquely hamiltonian graphs with girth at least 5 and minimum degree at least 3 is shown in Table~\ref{table:counts_mindeg3_girth5}. These graphs can be downloaded from the \textit{House of Graphs}~\cite{hog} at \url{http://hog.grinvin.org/UHG}. We also verified that there are no uniquely hamiltonian graphs with girth 4 and minimum degree at least 3 on 18 (or fewer) vertices.

\begin{table}[h!tb]
	\centering
	\small
	\begin{tabular}{lrrrrrr}
		\toprule	
		Order & $< 18$ & $18$ & $19$ & $20$ & $21$ & $22$\\
		Number of graphs & $0$ & $2$ & $1$ & $2$ & $25$ & $33$\\
		\bottomrule		
	\end{tabular}
	\caption{The number of uniquely hamiltonian graphs with girth at least 5 and minimum degree at least 3. All of these graphs have girth 5 and minimum degree 3.}
	\label{table:counts_mindeg3_girth5}
\end{table}

By extending our algorithm from Section~\ref{section:generation_algo} to only generate uniquely hamiltonian graphs without $(n - 1)$-cycles, we were also able to look into Royle's question which led to the following observation.

\begin{observation}
There are no uniquely hamiltonian graphs of minimum degree at least~$3$ without $(n - 1)$-cycles up to order~$18$. Furthermore, there are no uniquely hamiltonian graphs of minimum degree at least~$3$ and girth at least $4$ (respectively $5$) without $(n - 1)$-cycles up to order~$19$ (respectively $23$).
\end{observation}

\subsubsection{Sheehan's conjecture}

We recall a conjecture of Sheehan from 1974:

\bigskip

\noindent \textbf{Conjecture} (Sheehan~\cite{sheehan1975multiplicity}). \emph{There is no uniquely hamiltonian $4$-regular graph.}\hfill $({\mathfrak S})$

\bigskip

Petersen's 2-Factor Theorem~\cite{petersen1891theorie} states that every $(2k)$-regular graph can be decomposed into $k$ edge-disjoint 2-factors. Hence, if $({\mathfrak S})$ is true, then the only regular uniquely hamiltonian graphs are cycles.

In~\cite{haythorpe2017minimum} Haythorpe computationally determined the minimum number of hamiltonian cycles among all hamiltonian $k$-regular graphs of a given small order for $4 \leq k \leq 7$. In particular, he determined the minimum non-zero number of hamiltonian cycles in 4-regular hamiltonian graphs up to 16~vertices and in 4-regular  hamiltonian graphs of connectivity~2 up to 18~vertices.

Using the generator for regular graphs \textit{genreg}~\cite{meringer1999fast} we generated all 4-regular graphs up to 21~vertices and determined the minimum number of hamiltonian cycles for each order. The results, together with the counts for girth at least 4 and 5, are listed in Table~\ref{table:counts_4reg}. Note the striking drop in the minimum non-zero number of hamiltonian cycles for 4-regular graphs of girth at least 4 for orders 12 and 20.

\begin{table}[h!tb]
	\centering
	\small
	\begin{tabular}{crrr}
		\toprule	
		\multirow{2}{*}{Order} & \multicolumn{3}{c}{Minimum number of hamiltonian cycles} \\
		& girth $\geq 3$ &  girth $\geq 4$  & girth $\geq 5$ \\		
		\midrule
5 & 12 (1) & - & - \\
6 & 16 (1) & - & - \\
7 & 23 (1) & - & - \\
8 & 29 (1) & 72 & - \\
9 & 36 (1) & - & - \\
10 & 36 (1) & 96 & - \\
11 & 48 (2) & 145 & - \\
12 & 60 (2) & 142 & - \\
13 & 72 (3) & 250 & - \\
14 & 72 (1) & 323 & - \\
15 & 72 (2) & 460 & - \\
16 & 72 (1) & 604 & - \\
17 & 96 (2) & 936 & - \\
18 & 108 (1) & 1 024 & - \\
19 & 144 (21) & 1 512 & 2 688 \\
20 & 144 (18) & 1 024 & 2 716 \\
21 & 144 (13)  & 1 600 & 3 657 \\
22 & ?  & ? & 5 589 \\
23 & ?  & ? & 8 382 \\
24 & ?  & ? & 12 412 \\
25 & ?  & ? & 18 906 \\
26 & ?  & ? & 25 299 \\
		\bottomrule		
	\end{tabular}
	\caption{The minimum number of hamiltonian cycles among all 4-regular hamiltonian graphs of girth at least 3, 4, and 5 for that order. The symbol ``-" indicates that no 4-regular graphs of that order and girth exist. The numbers in parentheses indicate how many graphs of that order have the minimum non-zero number of hamiltonian cycles. For the columns regarding girth at least 4 and 5, respectively, there is in each case only one graph with the minimum non-zero number of hamiltonian cycles.}
	\label{table:counts_4reg}
\end{table}

Our results from Table~\ref{table:counts_4reg} imply the following.

\begin{observation}\label{obs:conj_sheehan}
Sheehan's conjecture $({\mathfrak S})$ is true for graphs on up to $21$~vertices. Furthermore, the conjecture is true for $4$-regular graphs with girth at least $5$ up to $26$~vertices.
\end{observation}



Haythorpe~\cite{haythorpe2017minimum} describes an infinite family of 4-regular graphs based on $K_5$ and the 1-skeleton of the octahedron. With a few minor additions we obtain the following result, which essentially belongs to him:

\begin{observation}
There exists a $4$-regular graph of order $n$ with exactly $36 \cdot 2^{\lfloor{\frac{n}{5}}\rfloor - 2}$ hamiltonian cycles for every $n \in \{ 10, 15, 16, 20, 21, 22, 25, 26, 27, 28 \}$ and all $n \ge 30$.
\end{observation}





\subsection{Cubic graphs with exactly three hamiltonian cycles}
\label{subsect:results_thg}

It follows from Smith's result that a hamiltonian cubic graph contains at least three hamiltonian cycles. We devote this section to the study of the extremal case of cubic graphs containing precisely three hamiltonian cycles, the smallest of which is $K_4$. Note that these three hamiltonian cycles together cover each edge exactly twice and thus form a \emph{cycle double cover} with the smallest possible number of cycles. For a brief overview of results on cubic graphs with exactly three hamiltonian cycles, see~\cite{holton1999planar}. In such a graph $G$, for each vertex $v$ in $G$ any pair of distinct edges incident with $v$ is traversed by exactly one hamiltonian cycle of $G$. By starting with $K_4$ and replacing vertices by triangles one can construct for every $k \ge 2$ a planar 3-connected cubic graph of order $n = 2k$ with exactly three hamiltonian cycles.

These graphs clearly have girth~3. Are there triangle-free graphs with exactly three hamiltonian cycles? (Note that we do not require planarity at this point.) In order to address this question we make use of ideas of Chia and Yu~\cite{chia1995number}, which we now briefly introduce. Clearly, between any two vertices of a triangle $T$ in a cubic graph there is exactly one path visiting every vertex of the triangle, and $T$ is separated from the rest of the graph by a 3-edge-cut. Thus, we can replace $T$ with a so-called ``tup'', a graph introduced in~\cite{chia1995number}: A \emph{tup} is a graph in which all but three vertices are cubic, the three exceptional vertices have degree~2, and between any two such vertices there is exactly one hamiltonian path.

One can make from any cubic graph with three hamiltonian cycles a tup by removing one of the graph's vertices. Given a cubic graph $G$ and a vertex $v$ in $G$, one can replace $v$ with a tup $H$, where each neighbour of $v$ in $G$ is connected by an edge to one of the 2-valent vertices of $H$. The resulting graph will have the same number of hamiltonian cycles as $G$. In fact $G - v$ is a tup. More generally, when we consider two tups $H$ and $H'$ and join their 2-valent vertices by edges using a bijection, we say that we \emph{merge} $H$ and $H'$ and write $H \circ H'$ for the graph we obtain. When a vertex $v$ and a tup $H$ are merged, we are referring to the graph $H \circ v = (V(H) \cup \{ v \}, E(H) \cup \{ vx_1, vx_2, vx_3 \})$, where $x_1, x_2, x_3$ are the 2-valent vertices of $H$.

Chia and Yu~\cite{chia1995number} proved that a cubic graph $G$ has exactly three hamiltonian cycles if and only if $G$ is the merger of two tups, or of a vertex and a tup.
We shall use the family of {\em generalised Petersen graphs} $${\rm GP}(n,k) = \left( \{u_i, u'_i \}_{i = 0}^{n-1}, \{ u_i u_{i+1}, u_i u'_i, u'_i u'_{i+k} \}_{i=0}^{n-1} \right)\hspace{-1mm},$$ with indices mod $n$ and $k < n/2$.


Schwenk~\cite{schwenk1989enumeration} proved that $h({\rm GP}(n,2)) = 3$ if and only if $n \equiv 3 \ {\rm mod} \ 6$. These graphs have girth 5, so they yield an infinite family of triangle-free cubic graphs with exactly three hamiltonian cycles. The question now is whether there are other such graphs---the answer is \emph{yes, but no small ones}:

\begin{theorem}
The only triangle-free cubic graphs of order at most $32$ with exactly three hamiltonian cycles are ${\rm GP}(9,2)$ and ${\rm GP}(15,2)$. However, for every $k \ge 19$ there exists a cubic graph of order $2k$ and girth~$4$ having exactly three hamiltonian cycles, and for every $\ell \ge 17$ there exists a cubic graph of order $2\ell$ and girth~$5$ with exactly three hamiltonian cycles.
\end{theorem}

\begin{proof}
For the first statement we ran the program \textit{snarkhunter}~\cite{brinkmann2017generation,brinkmann_11} to generate all cubic graphs up to 32~vertices and used a separate program to count the number of hamiltonian cycles of the generated graphs. This allowed us to determine all cubic graphs with exactly 0, 3, or $>3$ hamiltonian cycles up to 32~vertices, see Tables~\ref{table:counts_num_cycles_cubic}--\ref{table:counts_num_cycles_cubic-g5} in the Appendix.

We now show the third statement. For $k \equiv 3 \ {\rm mod} \ 6$ we know that ${\rm GP}(k,2)$ is a cubic graph of order $2k$ and girth 5 which has exactly three hamiltonian cycles. Removing a vertex $v$ from ${\rm GP}(9,2)$ yields the tup $G_1$ of order 17 with 2-valent vertices $x,y,z$. Moreover, consider the tups $$G_2 = (V(G_1) \cup \{ v_1, v_2 \}, E(G_1) \cup \{ x v_1, v_1 v_2, v_2 y \}),$$ $$G_3 = (V(G_2) \cup \{ v_3, v_4 \}, E(G_2) \cup \{ v_2v_3, v_3v_4, v_4z \}),$$ and $$G_4 = (V(G_3) \cup \{ v_5, v_6 \}, E(G_3) \cup \{ v_4v_5, v_5v_6, v_6v_1 \}).$$
These graphs have order 19, 21, and 23, respectively.

Merging the tup $G_1$ with the tups $G_1, G_2, G_3, G_4$ yields graphs of order $34, 36, 38, 40$. For order $42$, consider ${\rm GP}(21,2)$. For order 44, merge $G_3$ and $G_4$ carefully (to ensure that the resulting graph has girth~5): if $v_1, v_3, v_4$ are the exceptional 2-valent vertices of $G_3$ and $v'_3$, $v'_5$, $v'_6$ are the exceptional 2-valent vertices of $G_4$, then join $v_1$ with $v'_6$, $v_3$ with $v'_3$, and $v_4$ with $v'_5$. For order 46, merge two copies $G_4$: if $v_3, v_5, v_6$ are the exceptional 2-valent vertices of the first copy and $v'_3$, $v'_5$, $v'_6$ are the exceptional 2-valent vertices of the second copy, then join $v_3$ with $v'_5$, $v_5$ with $v'_3$, and $v_6$ with $v'_6$. For order 48, consider ${\rm GP}(15,2)$, remove a vertex, and apply the same operation as above with which we obtained $G_2$ from $G_1$. We obtain a 31-vertex tup which we merge with $G_1$, which has order 17, yielding our desired graph. We have described cubic graphs $T_n$ of girth~5 containing exactly three hamiltonian cycles---the straightforward verification of the details, in particular concerning the girth requirement, are left to the reader---of all even orders $n$ in $34, \ldots, 48$.

For a graph $G$ we denote by $G^*$ the removal of an arbitrary vertex of $G$. For orders $\ge 50$, consider $$\{ {\rm GP}(k,2)^* \circ T_n^* : k \equiv 3 \ {\rm mod} \ 6, \ n \in \{ 34, 36, 38, 40, 42, 44 \} \}.$$
Similar techniques yield the theorem's second statement.
\end{proof}





It remains an open question whether there exist cubic graphs of girth~4 with exactly three hamiltonian cycles that have order~34 or 36.

Using the program \textit{snarkhunter}~\cite{brinkmann2017generation,brinkmann_11} we also verified that there are no cubic graphs with exactly three hamiltonian cycles of girth~6 up to order 36, none of girth~7 up to order 40, none of girth~8 up to order 46, and none of girth~9 up to order~64.

\begin{theorem}
Let $G$ be a graph in which all vertices have odd degree, and containing exactly $p$ hamiltonian cycles, where $p$ is a prime number. Then $G$ is $3$-connected.
\end{theorem}

\begin{proof}
Since $p \ge 2$, $G$ is hamiltonian, so it must be 2-connected. Suppose $G$ has connectivity~2. Then it contains a 2-cut $X = \{ x,y \}$ whose removal from $G$ yields exactly two components $C, C'$. (If more than two components are present, we obtain a contradiction to the hamiltonicity of $G$ by a simple toughness argument.) Denote the degrees of $x$ and $y$ in $F = G[V(C) \cup X]$ by $d_x$ and $d_y$, respectively. We can assume that either $F$ or $G[V(C') \cup X]$, say $F$, contains exactly one hamiltonian $xy$-path: if they would both contain more than one such path, say $k \ge 2$ and $k' \ge 2$ respectively, then $G$ would contain $kk'$ hamiltonian cycles, but this is impossible as $G$ contains a prime number of hamiltonian cycles. Three situations can occur:

\smallskip

\noindent \textsc{Case 1.} \emph{Both $d_x$ and $d_y$ are even.} Consider the graph $$G_1 = (V(F) \cup \{ v_1, v_2, v_3, v_4 \}, E(F) \cup \{ xv_1, v_1v_2, v_1v_3, v_2v_3, v_2v_4, v_3v_4, v_4y \} ).$$ All vertices in $G_1$ have odd degree and $G_1$ contains exactly two hamiltonian cycles, as $F$ contains exactly one hamiltonian $xy$-path and $G_1[\{ x, y, v_i \}_{i=1}^4]$ contains exactly two hamiltonian $xy$-paths.

\smallskip

\noindent \textsc{Case 2.} \emph{$d_x$ is even while $d_y$ is odd.} Consider a copy $F'$ of $F - xy$ and denote the copy of $x$ ($y$) in $F'$ by $x'$ ($y'$). In $(F - xy) \cup F'$, identify $x$ with $y'$ and $y$ with $x'$. Denote the graph we obtain by $G_2$. All vertices in $G_2$ have odd degree and $G_2$ is uniquely hamiltonian.

\smallskip

\noindent \textsc{Case 3.} \emph{Both $d_x$ and $d_y$ are odd.} Let $$G_3 = (V(F) \cup \{ v_1, v_2 \}, E(F) \cup \{ xv_1, xv_2, v_1v_2, v_1y, v_2y \} ).$$ The graph $G_3$ contains only vertices of odd degree and $G_3$ contains exactly two hamiltonian cycles.

\smallskip

Each of the above cases leads to a contradiction, since by Thomason's theorem a hamiltonian graph in which all vertices are of odd degree contains at least three hamiltonian cycles.
\end{proof}

\noindent \textbf{Conjecture} (Cantoni~\cite{Tutte79}). \emph{Every planar cubic graph with exactly three hamiltonian cycles contains a triangle.} \hfill $({\mathfrak C})$

\bigskip

Every planar cubic graph $G$ with exactly three hamiltonian cycles has connectivity~$3$ by above theorem. Furthermore, every vertex-deleted subgraph of $G$ must be hamiltonian---in particular, $G$ cannot be bipartite---by Thomason's Corollary~1.5 in~\cite{Th78} which states that in any cubic graph $H$ we have $h(H - v) = h(H) \ {\rm mod} \ 2$ for every vertex $v$ in $H$.

It was shown by Fowler (see page 30 of~\cite{fowler}) that if $G$ is a planar cubic graph with exactly three hamiltonian cycles, then the statement ``$G$ contains a triangle'', i.e.\ $({\mathfrak C})$, is equivalent to the statement ``$G$ is uniquely edge-3-colourable.'' For further definitions and details we refer to Fowler's thesis.

Using the program \textit{plantri}~\cite{brinkmann_07} we generated all planar $3$-connected triangle-free cubic graphs up to 48~vertices and tested if any such graph contains exactly three hamiltonian cycles. This resulted in the following observation.

\begin{observation}\label{obs:conj_cantoni}
Cantoni's conjecture $({\mathfrak C})$ is true up to at least $48$ vertices. Furthermore, there are no planar $3$-connected cubic graphs of girth $5$ with exactly three hamiltonian cycles up to at least $68$~vertices. 
\end{observation}

\subsection{On a question of Chia and Thomassen}
\label{subsect:chia_thomassen}

In contrast to Lemma~\ref{Thom-lemma}, it was shown by Chia and Yu~\cite{chia1995number} that for every $k \ge 3$ there exists a planar cyclically 3-edge-connected cubic graph with precisely $k$ hamiltonian cycles. 
The argument is short: The 1-skeleta of prisms, i.e.\ the cartesian product of a cycle with $K_2$, deal with all $k \ne 4$, while carefully combining two Tutte-fragments (first used in~\cite{Tu46}) yields the $k = 4$ case (we note that this construction has cyclical edge-connectivity~3). Many non-hamiltonian planar 3-connected cubic graphs are known---take for instance Thomassen's infinite family~\cite{thomassen1981planar}. In fact, Thomassen's graphs as well as all prisms excluding the triangular one are cyclically 4-edge-connected.

For $k = 4$ and cyclical edge-connectivity~4, we used the program \textit{plantri}~\cite{brinkmann_07} to generate all planar cyclically 4-edge-connected cubic graphs up to 48~vertices and tested if any such graph contains exactly four hamiltonian cycles. The result is as follows.

\begin{observation} \label{obs:pl_4HC}
The smallest planar cyclically $4$-edge-connected cubic graph with exactly four hamiltonian cycles has $38$ vertices and is shown in Figure~\ref{fig:planar_38v_4HC}. There are also exactly five such graphs on $42$ vertices, $32$ on $46$ vertices and six on $48$ vertices. These constitute all such graphs of order at most $48$.
\end{observation}

\begin{figure}[!htb]
\begin{center}
 \includegraphics[width=0.3\textwidth]{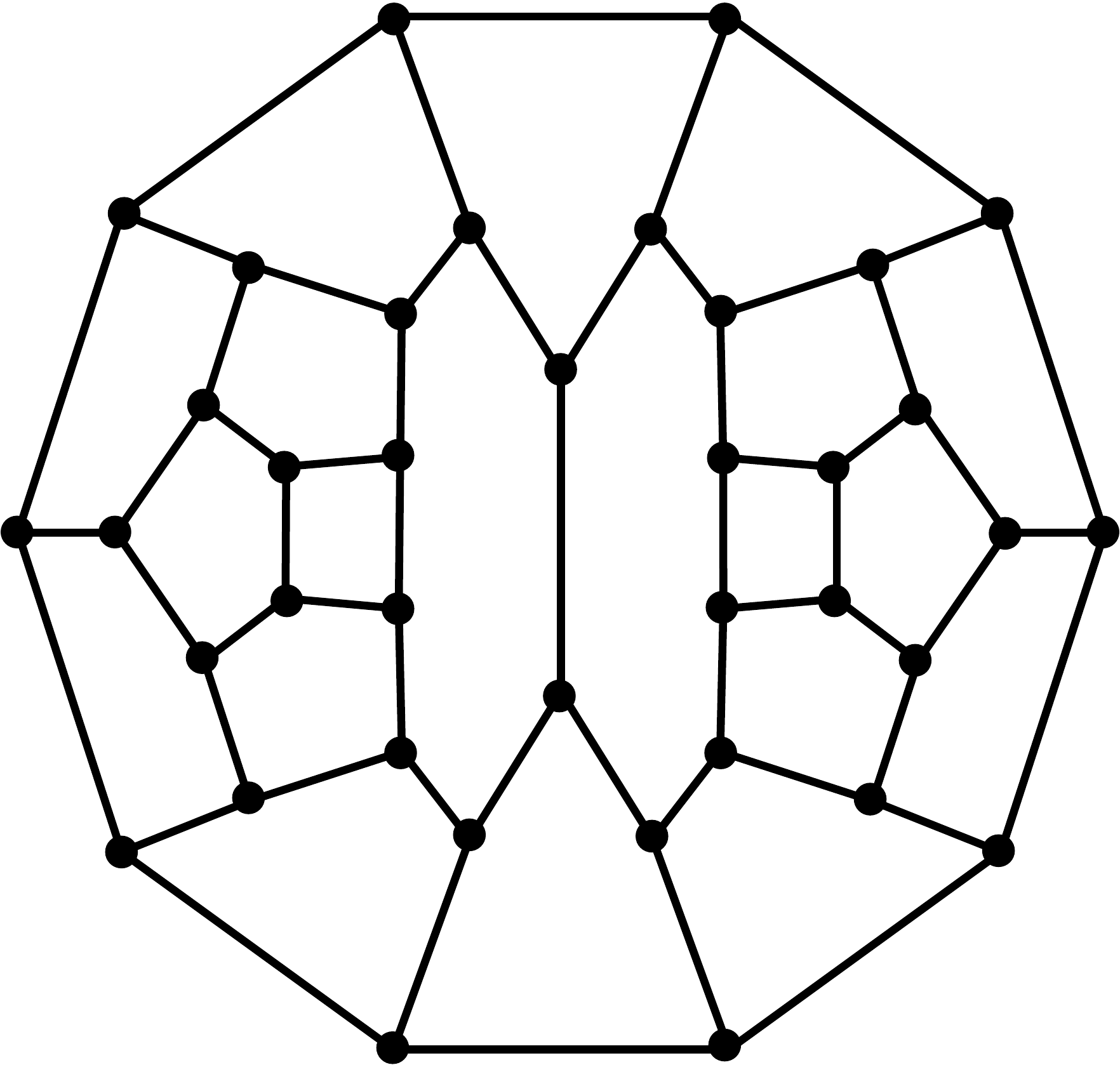}
	\caption{The smallest planar cyclically $4$-edge-connected cubic graph with exactly four hamiltonian cyles. It has 38 vertices.} 
	\label{fig:planar_38v_4HC}
\end{center}
\end{figure}

Drawings of the five planar cyclically $4$-edge-connected cubic graphs with exactly four hamiltonian cycles on $42$ vertices from Observation~\ref{obs:pl_4HC} can be found in Figure~\ref{fig:planar_42v_4HC} in the Appendix. Hence, together with Lemma~\ref{Thom-lemma}, we have:

\begin{theorem}
For every non-negative integer $k \notin \{ 1, 2, 3\}$ there exists a planar cyclically $4$-edge-connected cubic graph with exactly $k$ hamiltonian cycles, while for $k \in \{ 1, 2 \}$ there exist no such graphs with precisely $k$ hamiltonian cycles. 
\end{theorem}

Thus only the case $k = 3$ remains open, which relates to a conjecture due to Cantoni on which we focus in Section~\ref{subsect:results_thg}. We remark that if one subdivides one of the quadrilaterals present in the graph depicted in Figure~\ref{fig:planar_38v_4HC} into an odd number of quadrilaterals by adding an even number $2k$ of (parallel) edges---we have illustrated the result of this operation for $k = 1$ in Figure~\ref{fig:planar_42v_4HC_subdiv} in the Appendix---, it follows that for every $\ell \ge 0$ there exists a planar cyclically $4$-edge-connected cubic graph of order $38 + 4\ell$ containing exactly four hamiltonian cycles. One can apply the same procedure to the quadrilateral emphasised in the 48-vertex graph shown in Figure~\ref{fig:planar_48v_4HC_ladder_base}, which yields the 52-vertex graph shown in Figure~\ref{fig:planar_48v_4HC_ladder_expanded} when performing the operation for $k=1$.

Together with Observation~\ref{obs:pl_4HC}, we obtain:

\begin{figure}[h!tb]
    \centering
   \subfloat[]{\label{fig:planar_48v_4HC_ladder_base}\includegraphics[width=0.28\textwidth]{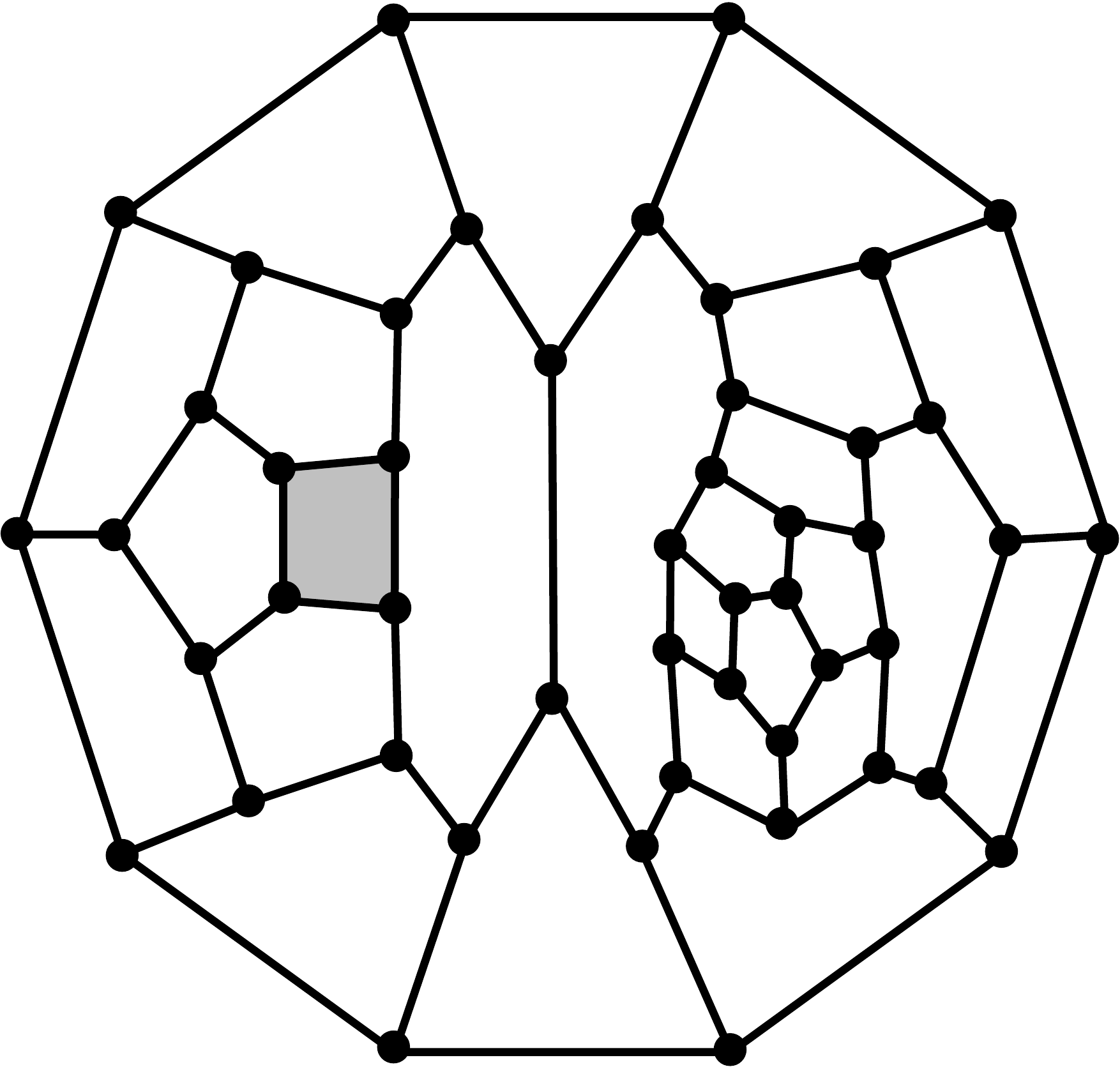}} \qquad \qquad
    \subfloat[]{\label{fig:planar_48v_4HC_ladder_expanded}\includegraphics[width=0.28\textwidth]{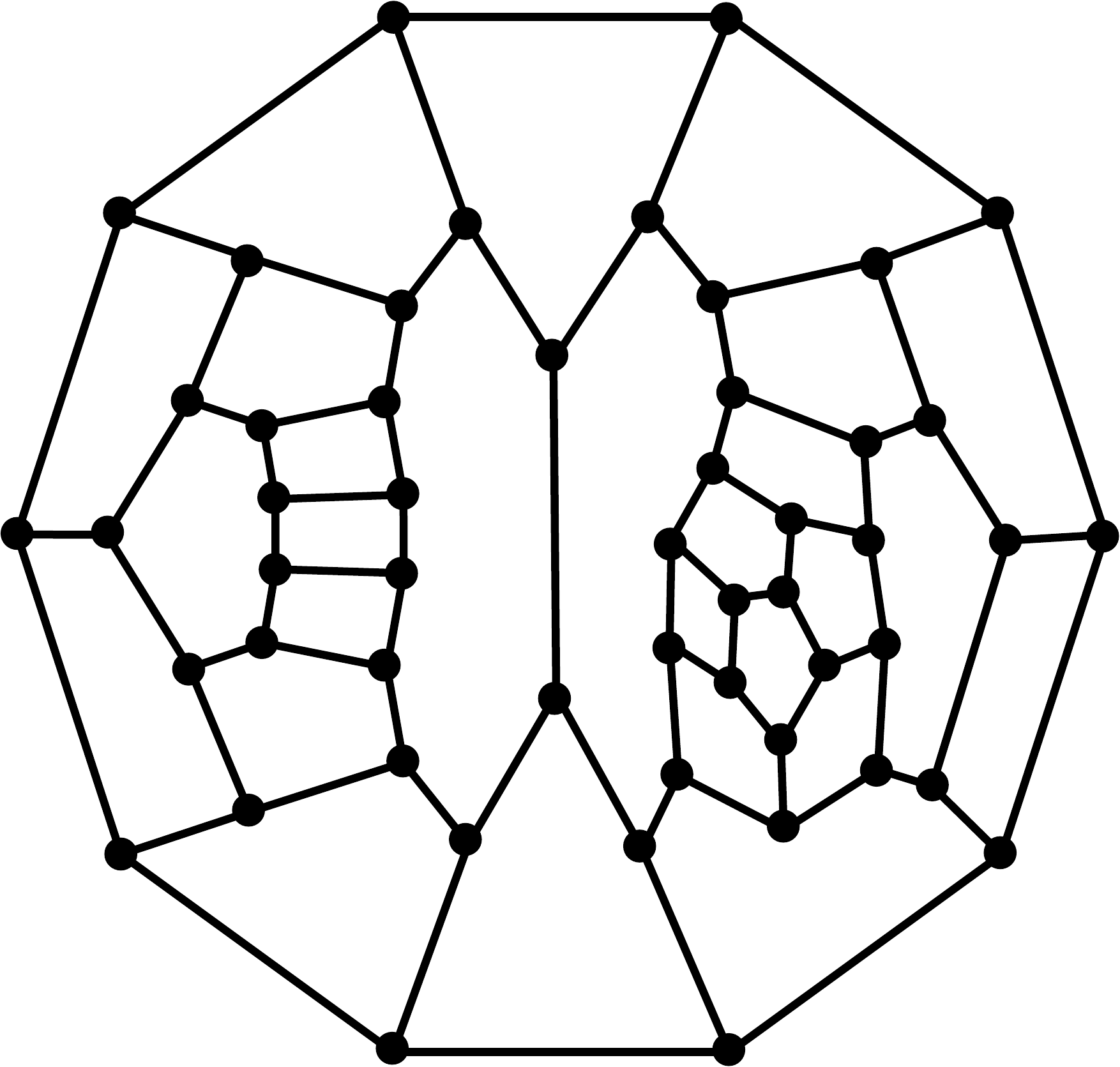}}
    \caption{A planar cyclically $4$-edge-connected cubic graph with exactly four hamiltonian cyles on 48 vertices (left-hand side) and the planar cyclically $4$-edge-connected cubic graph with exactly four hamiltonian cyles on 52 vertices obtained from it by subdividing the emphasised quadrilateral (right-hand side).}
	\label{fig:planar_48v_4HC_ladder}
\end{figure}

\begin{theorem}
There exists a planar cyclically $4$-edge-connected cubic graph of order~$n$ with exactly four hamiltonian cycles if and only if $n \in \{ 38, 42 \}$ or $n \ge 46$ is even.
\end{theorem}

This gives a (negative) answer to~\cite[Question 1]{chia2012} of Chia and Thomassen; they asked whether every planar cyclically 4-edge-connected cubic graph on $n$ vertices contains at least $n/2$ longest cycles. We remark that a similar result has recently (and independently) been obtained by Pivotto and Royle~\cite{pivotto2019}. For details on the operation adding edges to a quadrilateral used above, we refer to their article.

\subsection{Uniquely traceable graphs}
\label{subsect:results_utg}

Motivated by results on uniquely hamiltonian graphs, we now study a natural variation of the concept and call a graph \emph{uniquely traceable} if it contains exactly one hamiltonian path. Sheehan's~\cite[Theorem~1]{sheehan1977graphs} states that a uniquely hamiltonian graph of order $n \ge 3$ has size at most $\lfloor \frac{n^2}{4} \rfloor + 1$. This upper bound is sharp as Sheehan constructed for every $n \ge 3$ an $n$-vertex uniquely hamiltonian graph of maximum size. He claims in~\cite[Theorem~2]{sheehan1977graphs} that these are the only graphs reaching the bound. However, Barefoot and Entringer showed that for $n \ge 7$ there are exactly $2^{\lceil \frac{n}{2} \rceil - 4}$ such graphs~\cite{barefoot1981census}, so for uniquely hamiltonian graphs of order at most 8 Sheehan's Theorem~2 does hold, but for $n \ge 9$ it does not. Sheehan writes~\cite{sheehan1977graphs}: ``An analogous result to Theorem~1 can easily be obtained when instead of Hamiltonian circuits we consider Hamiltonian paths. In this case $h(n)$ [the maximum size of an $n$-vertex uniquely hamiltonian graph] is replaced by $[n^2/4]-[n/2]+ 1$, and the extremal graph is again unique.'' In the following we shall confirm Sheehan's size bound, but disprove his uniqueness claim.

\begin{observation}\label{obs:ut}
Let $G$ be a uniquely traceable graph. Then $G$ has exactly two vertices of degree~$1$, namely the starting point and endpoint of the hamiltonian path. In particular, $G$ has connectivity~$1$. Furthermore, identifying the endpoints of the hamiltonian path of $G$, we obtain a uniquely hamiltonian graph. Conversely, if a uniquely hamiltonian graph contains a vertex of degree~$2$, then splitting this vertex into two vertices of degree~$1$ yields a uniquely traceable graph.
\end{observation}

\begin{proof}
We denote the unique hamiltonian path of $G$ by $v_1 \ldots v_n$. Assume that the degree of $v_1$ is at least 2. Then $v_1$ is adjacent to $v_i$ for some $i \ge 3$. Thus $v_{i-1} v_{i-2} \ldots v_1 v_i v_{i+1} \ldots v_n$ is a second hamiltonian path in $G$, a contradiction. So $v_1$ and analogously $v_n$ must have degree~1.

Denote the graph resulting from the identification of $v_1$ and $v_n$ by $G'$, and let $v$ be the vertex in $G'$ obtained by identifying $v_1$ and $v_n$. Combining that $G' - v$ contains exactly one hamiltonian $v_2 v_{n-1}$-path (otherwise $G$ would not be uniquely traceable) with the fact that, due to the argument above, $v$ has degree~2, the graph $G'$ is uniquely hamiltonian. The converse argument is very similar.
\end{proof}

\begin{theorem}
Let $n \ge 2$. Then the following hold.
\begin{enumerate}
\item A uniquely traceable graph of order $n$ has at most $s(n) = \lfloor \frac{(n - 1)^2}{4} \rfloor + 1$ edges.
\item For every $m$ such that $n - 1 \le m \le s(n)$ there exists a uniquely traceable graph of order $n$ and size $m$.
\item There exist exactly $\max\left\{1, 2^{\lceil \frac{n-1}{2} \rceil - 3} \right\}$ uniquely traceable graphs of order $n$ and size $s(n)$.
\end{enumerate}
\end{theorem}

\begin{proof}
Barefoot and Entringer~\cite{barefoot1981census} showed that for every $n \ge 7$ there exist uniquely hamiltonian graphs of order $n$ and size $s(n+1)$, and no such graphs of greater size. Statement~1 now follows: for $n < 8$ it is elementary to verify, so consider henceforth $n \ge 8$. Assume that there exists an $n$-vertex uniquely traceable graph $G$ of size greater than $s(n)$. Using Observation~\ref{obs:ut}, if we identify the 1-valent vertices of $G$, we obtain a uniquely hamiltonian graph of order $n - 1$ and size greater than $s(n)$. This however contradicts the result of Barefoot and Entringer.

They also describe the following crucial property of uniquely hamiltonian graphs of order $n$ and size $s(n+1)$: there is always a 5-vertex subpath of the hamiltonian cycle such that the degrees of these vertices are $2, n-1, 2, n-2, 3$, respectively, and all other vertices have degree more than 2 and less than $n - 2$. Thus, if we use the two vertices of degree~2 to produce uniquely traceable graphs (see Observation~\ref{obs:ut}), we obtain two non-isomorphic graphs by a valency argument. Two uniquely traceable graphs that are constructed from non-isomorphic uniquely hamiltonian graphs of order $n$ and size $s(n+1)$ must themselves be non-isomorphic. Barefoot and Entringer showed that there are exactly $2^{\lceil \frac{n}{2} \rceil - 4}$ uniquely hamiltonian graphs of order $n$ and maximum size. We thus obtain $2^{\lceil \frac{n}{2} \rceil - 3}$ uniquely traceable graphs of order $n + 1$ and maximum size. Consider such a uniquely traceable graph $G$ of maximum size. In $G$, we call edges not belonging to its unique hamiltonian path \emph{chords}. Removing chords from $G$ one-by-one yields statement 2.

For statement 3, suppose that there are more than $2^{\lceil \frac{n-1}{2} \rceil - 3}$ uniquely traceable graphs of order $n$. This means that there exists a uniquely traceable graph $G$ of order $n$ and size $s(n)$ which was not obtained from an $(n-1)$-vertex uniquely hamiltonian graph $G'$ of size $s(n)$ by splitting a vertex of degree~2 into two 1-valent vertices. However, identifying in $G$ its 1-valent vertices, by Observation~\ref{obs:ut} we obtain a uniquely hamiltonian graph of order $n - 1$ and, since identifying did not change the size, exactly $s(n)$ edges, a contradiction. Hence, each uniquely traceable graph of maximum size can be constructed from a uniquely hamiltonian graph of maximum size, and thus we find all uniquely traceable graphs of order $n \ge 8$ and maximum size from the uniquely hamiltonian graphs of order $n - 1$.

Finally, for $n \le 7$ the unique uniquely traceable graphs of order $n$ and size $s(n)$ are the paths for $n \le 4$ and the graphs shown in Figure~\ref{fig:ut} for $n \in \{ 5, 6, 7\}$. (These can be obtained from the uniquely hamiltonian graphs of order at most 6 and of maximum size.) \end{proof}

\begin{figure}[!ht]
\begin{center}
\includegraphics[height=15mm]{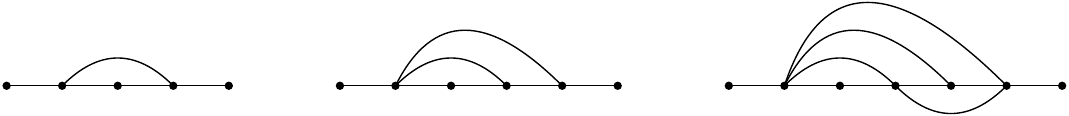}
	\caption{The unique uniquely traceable graphs of maximum size and order 5, 6, and 7.}
	\label{fig:ut}
\end{center}
\end{figure}

\subsection{Remarks on computational results}
\label{sect:comput_results}

We implemented the algorithm from Section~\ref{section:generation_algo} in C. Our program is called \textit{generateUHG} and its source code can be downloaded from~\cite{genUHG-site}. We used this program to generate complete lists of all non-isomorphic graphs with exactly $k$ hamiltonian cycles of a given order (and a given lower bound on the girth) for various values of $k$. The main computational results for uniquely hamiltonian graphs are listed in this section. (The counts of planar uniquely hamiltonian graphs were already reported in Table~\ref{table:counts_planar_uhg} from Section~\ref{subsect:results_uhg} in the context of the Bondy-Jackson conjecture.) Additional tables with counts of graphs with exactly $k$ hamiltonian cycles for $k \neq 1$ and counts for cubic graphs can be found in the Appendix.

Table~\ref{table:counts_uhg} lists the counts of uniquely hamiltonian graphs, such graphs of girth at least~4, and such graphs of girth at least~5. The running times of our program are reported in Table~\ref{table:runningtime_uhg}. The code was compiled using gcc and was performed on Intel Xeon E5-2680 CPU's at 2.60GHz. The running times for the larger orders include a small overhead due to parallelisation. 
The total computational effort for this project amounted to 40 CPU years and the most time-consuming computations were the generation of all uniquely hamiltonian graphs on 15~vertices and the test if any cubic graphs on 32~vertices contain exactly three hamiltonian cycles (which required 11 and 20 CPU years, respectively).

We also compared the running times of our program \textit{generateUHG} to the approach of using the program \textit{geng}~\cite{nauty-website, mckay_14} to generate all graphs and then filtering the uniquely hamiltonian graphs (which was, as far as we are aware, up until now the only available method in the literature to generate all uniquely hamiltonian graphs of a given order). Our program is about 200 times faster than the filter approach for generating uniquely hamiltonian graphs of order 11. For order 12 it is 600 times faster than the filter approach and for order 13 about 3000 times faster.

\begin{table}[h!tb]
	\centering
	\small
	\begin{tabular}{crrr}
		\toprule	
		Order & \#\,UH graphs &  girth $\geq 4$  & girth $\geq 5$ \\
		\midrule
		3 & 1 & 0  & 0 \\
		4 & 2 & 1 & 0 \\
		5 & 3 & 1  & 1 \\
		6 & 12 & 2 & 1 \\
		7 & 49 & 3 & 1  \\
		8 & 482 & 11 & 3 \\
		9 & 6 380 & 38  & 4  \\
		10 & 135 252 & 250  & 10  \\
		11 & 3 939 509 & 2 171  & 32  \\
		12 & 166 800 470 & 25 518  & 167 \\	
		13 & 9 739 584 172	& 388 854 & 899 \\
		14 & 818 717 312 364 & 7 283 110 & 6 470 \\
		15 & 95 353 226 103 276 & 171 355 621 & 55 815 \\	
		16 & ? & 4 915 591 680 & 549 981 \\
		17 & ? & 174 203 813 967 & 6 155 795 \\
		18 & ? & 7 526 329 299 531 & 78 520 177 \\
		19 & ? & ? & 1 123 544 810 \\
		20 & ? & ? & 18 005 054 988 \\
		21 & ? & ? & 322 434 738 089 \\
		22 & ? & ? & 6 427 598 615 569\\
		\bottomrule		
	\end{tabular}
	\caption{The number of uniquely hamiltonian graphs, uniquely hamiltonian graphs of girth at least~4, and uniquely hamiltonian graphs of girth at least~5, respectively.}
	\label{table:counts_uhg}
\end{table}

\begin{table}[h!tb]
	\centering
	\small
	\begin{tabular}{crrr}
		\toprule	
		Order & time UH graphs &  time girth $\geq 4$  & time girth $\geq 5$ \\
		\midrule
10 & $<1$ & $<1$ & $<1$ \\
11 & 8.2 & $<1$ & $<1$ \\
12 & 383 & $<1$ & $<1$ \\
13 & 25 944 & 1 & $<1$ \\
14 & 2 487 313 & 20.2 & $<1$ \\
15 & 358 436 755 & 527 & $<1$ \\
16 &  & 19 375 & 1.7 \\
17 &  & 739 167 & 21 \\
18 &  & 36 974 877 & 292 \\
19 &  &  & 5 242 \\
20 &  &  & 89 701 \\
21 &  &  & 1 722 169 \\
22 &  &  & 34 513 677 \\
		\bottomrule		
	\end{tabular}
	\caption{Running time (in seconds) of our algorithm for the generation of uniquely hamiltonian graphs, uniquely hamiltonian graphs of girth at least~4, and uniquely hamiltonian graphs of girth at least~5, respectively.}
	\label{table:runningtime_uhg}
\end{table}

The graphs from Tables~\ref{table:counts_planar_uhg} and~\ref{table:counts_uhg} can be downloaded from the \textit{House of Graphs}~\cite{hog} at \url{http://hog.grinvin.org/UHG}.


\subsubsection*{Correctness testing}

It is important to independently verify computational results to minimise the chance of programming errors. The counts of all (connected) non-hamiltonian graphs up to 12~vertices were already on the On-Line Encyclopedia of Integer Sequences~\cite{OEIS} (i.e.\ sequence A126149) and are in complete agreement with our results from Table~\ref{table:counts_num_cycles}.

We ran the program \textit{geng}~\cite{nauty-website, mckay_14} to generate all graphs up to 13~vertices and used a separate program to count the number of hamiltonian cycles of the generated graphs (cf.\ Table~\ref{table:counts_num_cycles} in the Appendix). Also here the results were in complete agreement with the counts we obtained using our generator for graphs with exactly $k$ hamiltonian cycles. Furthermore, we used \textit{geng} to compute all uniquely hamiltonian graphs of girth at least 4 up to 16 vertices and all uniquely hamiltonian graphs of girth at least 5 up to 20~vertices. These counts were in complete agreement with the results from Table~\ref{table:counts_uhg}, as well.

Similarly, we used the program \textit{snarkhunter}~\cite{brinkmann_11} to generate all cubic graphs up to 32~vertices and used a separate program to count the number of hamiltonian cycles of the generated graphs (cf.\ Table~\ref{table:counts_num_cycles_cubic}).
By restricting the maximum degree to three in our generator for graphs with $k$ hamiltonian cycles, we were able to verify these results up to 20~vertices.




\section*{Acknowledgements}

Most computations for this work were carried out using the Stevin Supercomputer Infrastructure at Ghent University. We would like to thank Gunnar Brinkmann for providing us with an independent program for counting hamiltonian cycles.



\bibliographystyle{plain}
\bibliography{references.bib}

\begin{thebibliography}{10}

\bibitem{abbasi2006degree}
S.~Abbasi and A.~Jamshed.
\newblock {A Degree Constraint for Uniquely Hamiltonian Graphs}.
\newblock {\em Graphs and Combinatorics}, 22(4):433--442, 2006.

\bibitem{barefoot1981census}
C.A. Barefoot and R.C. Entringer.
\newblock {A census of maximum uniquely hamiltonian graphs}.
\newblock {\em Journal of Graph Theory}, 5(3):315--321, 1981.

\bibitem{berge1973graphs}
C.~Berge.
\newblock {\em Graphs and {H}ypergraphs}.
\newblock North-Holland Publishing Company, 1973.

\bibitem{bondy1998vertices}
J.A. Bondy and B.~Jackson.
\newblock {Vertices of Small Degree in Uniquely Hamiltonian Graphs}.
\newblock {\em Journal of Combinatorial Theory, Series B}, 74(2):265--275,
  1998.

\bibitem{boyer2004cutting}
J.M. Boyer and W.J. Myrvold.
\newblock {On the Cutting Edge: Simplified $O(n)$ Planarity by Edge Addition}.
\newblock {\em Journal of Graph Algorithms and Applications}, 8(2):241--273,
  2004.

\bibitem{hog}
G.~Brinkmann, K.~Coolsaet, J.~Goedgebeur, and H.~M{\'e}lot.
\newblock {House of Graphs: a database of interesting graphs}.
\newblock {\em Discrete Applied Mathematics}, 161(1-2):311--314, 2013.
\newblock Available at \url{http://hog.grinvin.org/}.

\bibitem{brinkmann2017generation}
G.~Brinkmann and J.~Goedgebeur.
\newblock Generation of cubic graphs and snarks with large girth.
\newblock {\em Journal of Graph Theory}, 86(2):255--272, 2017.

\bibitem{brinkmann_11}
G.~Brinkmann, J.~Goedgebeur, and B.D. McKay.
\newblock Generation of cubic graphs.
\newblock {\em Discrete Mathematics and Theoretical Computer Science},
  13(2):69--80, 2011.

\bibitem{brinkmann_07}
G.~Brinkmann and B.D. McKay.
\newblock Fast generation of planar graphs.
\newblock {\em MATCH Communications in Mathematical and in Computer Chemistry},
  58(2):323--357, 2007.

\bibitem{chia2012}
G.L. Chia and C.~Thomassen.
\newblock {On the number of longest and almost longest cycles in cubic graphs}.
\newblock {\em Ars Combinatoria}, 104:307--320, 2012.

\bibitem{chia1995number}
G.L. Chia and R.Q.L. Yu.
\newblock On the number of {H}amilton cycles in cubic graphs.
\newblock {\em Congressus Numerantium}, 110:13--32, 1995.

\bibitem{entringer1980spanning}
R.C. Entringer and H.~Swart.
\newblock Spanning cycles of nearly cubic graphs.
\newblock {\em Journal of Combinatorial Theory, Series B}, 29(3):303--309,
  1980.

\bibitem{fleischner2014uniquely}
H.~Fleischner.
\newblock {Uniquely Hamiltonian Graphs of Minimum Degree 4}.
\newblock {\em Journal of Graph Theory}, 75(2):167--177, 2014.

\bibitem{fowler}
T.G. Fowler.
\newblock {\em Unique Coloring of Planar Graphs}.
\newblock PhD thesis, Georgia Institute of Technology, 1998.

\bibitem{genUHG-site}
J.~Goedgebeur, B.~Meersman, and C.T. Zamfirescu.
\newblock Homepage of genhypohamiltonian: \url{http://caagt.ugent.be/uhg/}.

\bibitem{grinberg1986three}
E.~Grinberg.
\newblock {Three-connected graphs with exactly one Hamiltonian cycle (in
  Russian)}.
\newblock {\em Republican Foundation of Algorithms and Programmes. Computing
  centre, P. Stutschka University, Riga, USSR}, 1986.

\bibitem{haxell2007independent}
P.~Haxell, B.~Seamone, and J.~Verstraete.
\newblock {Independent dominating sets and hamiltonian cycles}.
\newblock {\em Journal of Graph Theory}, 54(3):233--244, 2007.

\bibitem{haythorpe2017minimum}
M.~Haythorpe.
\newblock {On the Minimum Number of Hamiltonian Cycles in Regular Graphs}.
\newblock {\em Experimental Mathematics}, 27(4):426--430, 2018.

\bibitem{holton1999planar}
D.~Holton and R.E.L. Aldred.
\newblock {Planar Graphs, Regular Graphs, Bipartite Graphs and Hamiltonicity}.
\newblock {\em Australasian Journal of Combinatorics}, 20:111--132, 1999.

\bibitem{nauty-website}
B.D. McKay.
\newblock {nauty User's Guide (Version 2.5)}.
\newblock Technical Report TR-CS-90-02, Department of Computer Science,
  Australian National University. The latest version of the software is
  available at \url{http://cs.anu.edu.au/~bdm/nauty}.

\bibitem{mckay_98}
B.D. McKay.
\newblock {Isomorph-Free Exhaustive Generation}.
\newblock {\em Journal of Algorithms}, 26(2):306--324, 1998.

\bibitem{mckay_14}
B.D. McKay and A.~Piperno.
\newblock Practical graph isomorphism, {II}.
\newblock {\em Journal of Symbolic Computation}, 60:94--112, 2014.

\bibitem{meringer1999fast}
M.~Meringer.
\newblock Fast generation of regular graphs and construction of cages.
\newblock {\em Journal of Graph Theory}, 30(2):137--146, 1999.

\bibitem{petersen1891theorie}
J.~Petersen.
\newblock Die {T}heorie der regul{\"a}ren graphs.
\newblock {\em Acta Mathematica}, 15(1):193--220, 1891.

\bibitem{pivotto2019}
I.~Pivotto and G.~Royle.
\newblock Highly-connected planar cubic graphs with few or many {H}amilton
  cycles, \url{https://arxiv.org/abs/1901.10683}.

\bibitem{Ro17}
G.F. Royle.
\newblock The smallest uniquely hamiltonian graph with minimum degree at least
  3.
\newblock
  \url{https://mathoverflow.net/questions/255784/what-is-the-smallest-uniquely-hamiltonian-graph-with-minimum-degree-at-least-3/},
  2017.

\bibitem{schwenk1989enumeration}
A.J. Schwenk.
\newblock Enumeration of {H}amiltonian cycles in certain generalized {P}etersen
  graphs.
\newblock {\em Journal of Combinatorial Theory, Series B}, 47(1):53--59, 1989.

\bibitem{seamone2015uniquely}
B.~Seamone.
\newblock On uniquely {H}amiltonian claw-free and triangle-free graphs.
\newblock {\em Discussiones Mathematicae Graph Theory}, 35(2):207--214, 2015.

\bibitem{sheehan1975multiplicity}
J.~Sheehan.
\newblock The multiplicity of {H}amiltonian circuits in a graph.
\newblock In M.~Fiedler, editor, {\em Recent {A}dvances in {G}raph {T}heory},
  pages 477--480. Springer, 1975.

\bibitem{sheehan1977graphs}
J.~Sheehan.
\newblock Graphs with exactly one hamiltonian circuit.
\newblock {\em Journal of Graph Theory}, 1(1):37--43, 1977.

\bibitem{OEIS}
N.~Sloane.
\newblock The on-line encyclopedia of integer sequences:
  \url{http://oeis.org/}.

\bibitem{Th78}
A.G. Thomason.
\newblock {Hamiltonian Cycles and Uniquely Edge Colourable Graphs}.
\newblock In {\em Annals of Discrete Mathematics}, volume~3, pages 259--268.
  Elsevier, 1978.

\bibitem{thomassen1981planar}
C.~Thomassen.
\newblock Planar cubic hypohamiltonian and hypotraceable graphs.
\newblock {\em Journal of Combinatorial Theory, Series B}, 30(1):36--44, 1981.

\bibitem{thomassen1996number}
C.~Thomassen.
\newblock {On the Number of Hamiltonian Cycles in Bipartite Graphs}.
\newblock {\em Combinatorics, Probability and Computing}, 5(4):437--442, 1996.

\bibitem{thomassen1998independent}
C.~Thomassen.
\newblock {Independent Dominating Sets and a Second Hamiltonian Cycle in
  Regular Graphs}.
\newblock {\em Journal of Combinatorial Theory, Series B}, 72(1):104--109,
  1998.

\bibitem{Tu46}
W.T. Tutte.
\newblock {On Hamiltonian Circuits}.
\newblock {\em Journal of the London Mathematical Society}, 1(2):98--101, 1946.

\bibitem{Tutte79}
W.T. Tutte.
\newblock Hamiltonian circuits.
\newblock {\em Colloquio Internazionale sulle Teorie Combinatorie. Atti dei
  Convegni Lincei 17, Accad. Naz. Lincei, Roma I}, pages 193--199, 1976.

\end{thebibliography}

\newpage

\section*{Appendix}

\begin{table}[h!]
	\centering
	\setlength{\tabcolsep}{2.5pt}
	\small
	\begin{tabular}{c r r r r r r }
		\toprule	
		Order & $0$ & $1$ & $2$ & $3$ & $> 3$ \\
		\midrule
2 & 1 & 0 & 0 & 0 & 0 \\
3 & 1 & 1 & 0 & 0 & 0 \\
4 & 3 & 2 & 0 & 1 & 0 \\
5 & 13 & 3 & 2 & 0 & 3 \\
6 & 64 & 12 & 11 & 3 & 22 \\
7 & 470 & 49 & 75 & 17 & 242 \\
8 & 4 921 & 482 & 740 & 283 & 4 691 \\
9 & 83 997 & 6 380 & 10 692 & 5 069 & 154 942 \\
10 & 2 411 453 & 135 252 & 229 068 & 132 345 & 8 808 453 \\
11 & 123 544 541 & 3 939 509 & 7 005 022 & 4 451 059 & 867 760 434 \\
12 & 11 537 642 646 & 166 800 470 & 305 866 545 & 209 875 768 & 151 839 645 047 \\
13 & 2 013 389 528 672 & 9 739 584 172 & 18 868 736 922 & 13 364 007 134 & 48 280 546 012 319 \\
14 & ? & 818 717 312 364 & ? & ? & ? \\
15 & ? & 95 353 226 103 276 & ? & ? & ? \\
		\bottomrule
	\end{tabular}
	\caption{Number of connected graphs which contain exactly 0, 1, 2, 3 or $>3$ hamiltonian cycles, respectively.}
	\label{table:counts_num_cycles}
\end{table}

\begin{table}[h!]
	\centering
	\small
	\begin{tabular}{c r r r }
		\toprule	
		Order & $0$ & $3$ & $> 3$ \\
		\midrule
4 & 0 & 1 & 0 \\
6 & 0 & 1 & 1 \\
8 & 0 & 1 & 4 \\
10 & 2 & 3 & 14 \\
12 & 5 & 7 & 73 \\
14 & 35 & 24 & 450 \\
16 & 219 & 93 & 3 748 \\
18 & 1 666 & 435 & 39 200 \\
20 & 14 498 & 2 112 & 493 879 \\
22 & 148 790 & 11 019 & 7 159 638 \\
24 & 1 768 732 & 58 833 & 116 112 970 \\
26 & 24 029 714 & 322 733 & 2 070 128 417 \\
28 & 366 939 032 & 1 799 413 & 40 128 399 566 \\
30 & 6 213 299 362 & 10 185 443 & 839 256 743 264 \\
32 & 115 388 854 837 & 58 344 442 & 18 826 074 985 311 \\
		\bottomrule
	\end{tabular}
	\caption{Number of connected cubic graphs which contain exactly 0, 3 or $> 3$ hamiltonian cycles, respectively.}
	\label{table:counts_num_cycles_cubic}
\end{table}

\begin{table}[h!]
	\centering
	\small
	\begin{tabular}{c r r r }
		\toprule	
		Order & $0$ & $3$ & $> 3$ \\
		\midrule
4 & 0 & 0 & 0 \\
6 & 0 & 0 & 1 \\
8 & 0 & 0 & 2 \\
10 & 1 & 0 & 5 \\
12 & 0 & 0 & 22 \\
14 & 2 & 0 & 108 \\
16 & 8 & 0 & 784 \\
18 & 59 & 1 & 7 745 \\
20 & 425 & 0 & 97 121 \\
22 & 3 862 & 0 & 1 431 858 \\
24 & 41 293 & 0 & 23 739 521 \\
26 & 518 159 & 0 & 432 239 409 \\
28 & 7 398 734 & 0 & 8 535 072 760 \\
30 & 117 963 348 & 1 & 181 374 174 463 \\
32 & 2 069 516 990 & 0 & 4 125 007 626 872 \\
		\bottomrule
	\end{tabular}
	\caption{Number of connected cubic graphs with girth at least 4 which contain exactly 0, 3 or $> 3$ hamiltonian cycles, respectively.}
	\label{table:counts_num_cycles_cubic-g4}
\end{table}

\begin{table}[h!]
	\centering
	\small
	\begin{tabular}{c r r r }
		\toprule	
		Order & $0$ & $3$ & $> 3$ \\
		\midrule
4 & 0 & 0 & 0 \\
6 & 0 & 0 & 0 \\
8 & 0 & 0 & 0 \\
10 & 1 & 0 & 0 \\
12 & 0 & 0 & 2 \\
14 & 0 & 0 & 9 \\
16 & 0 & 0 & 49 \\
18 & 3 & 1 & 451 \\
20 & 15 & 0 & 5 768 \\
22 & 110 & 0 & 90 828 \\
24 & 1 130 & 0 & 1 619 349 \\
26 & 15 444 & 0 & 31 463 140 \\
28 & 239 126 & 0 & 656 544 764 \\
30 & 4 073 824 & 1 & 14 617 797 379 \\
32 & 75 458 941 & 0 & 345 900 189 621 \\
		\bottomrule
	\end{tabular}
	\caption{Number of connected cubic graphs with girth at least 5 which contain exactly 0, 3 or $> 3$ hamiltonian cycles, respectively.}
	\label{table:counts_num_cycles_cubic-g5}
\end{table}

\begin{figure}[h!tb]
    \centering
   \subfloat[]{\includegraphics[width=0.28\textwidth]{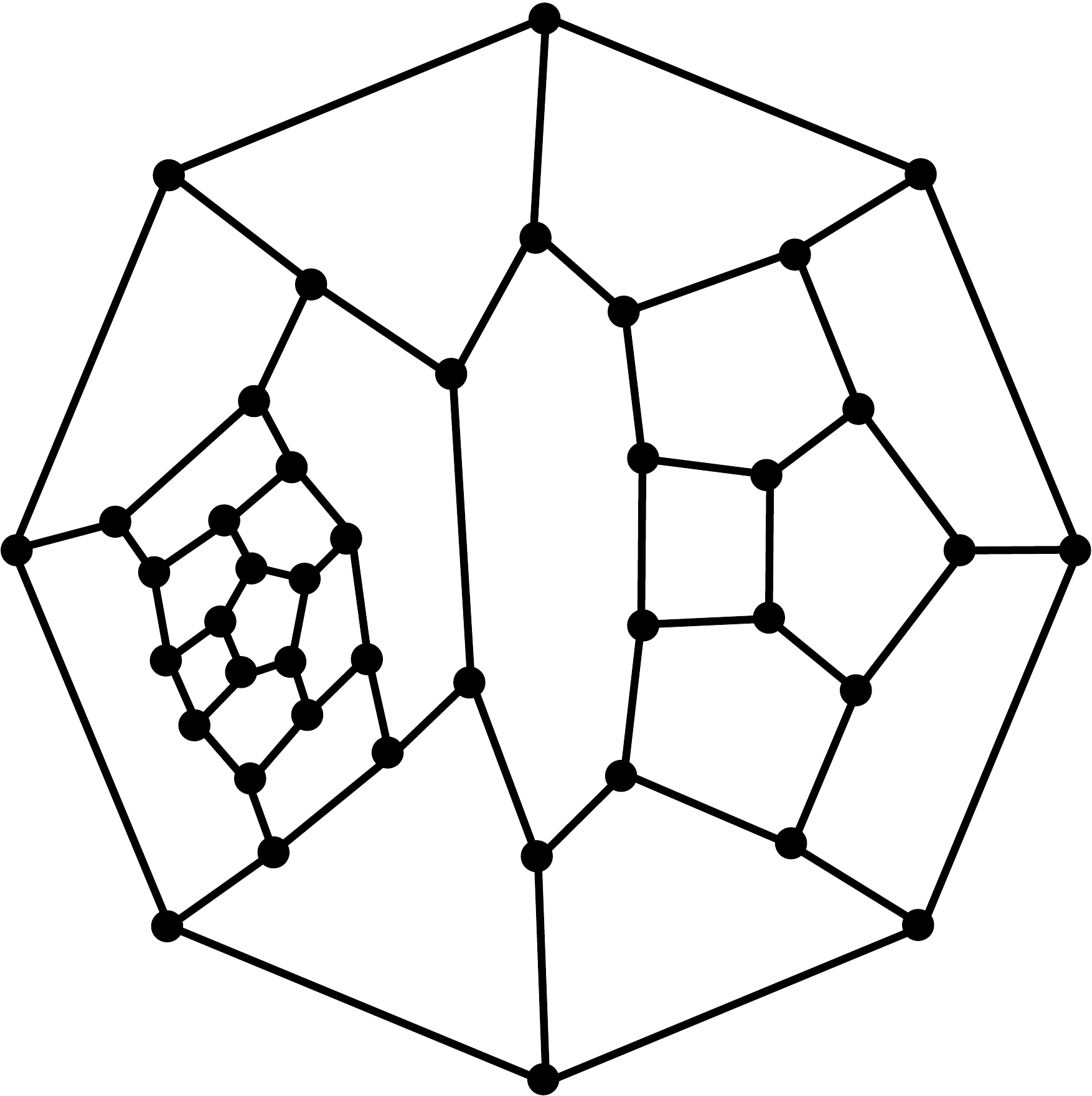}} \qquad
    \subfloat[]{\includegraphics[width=0.28\textwidth]{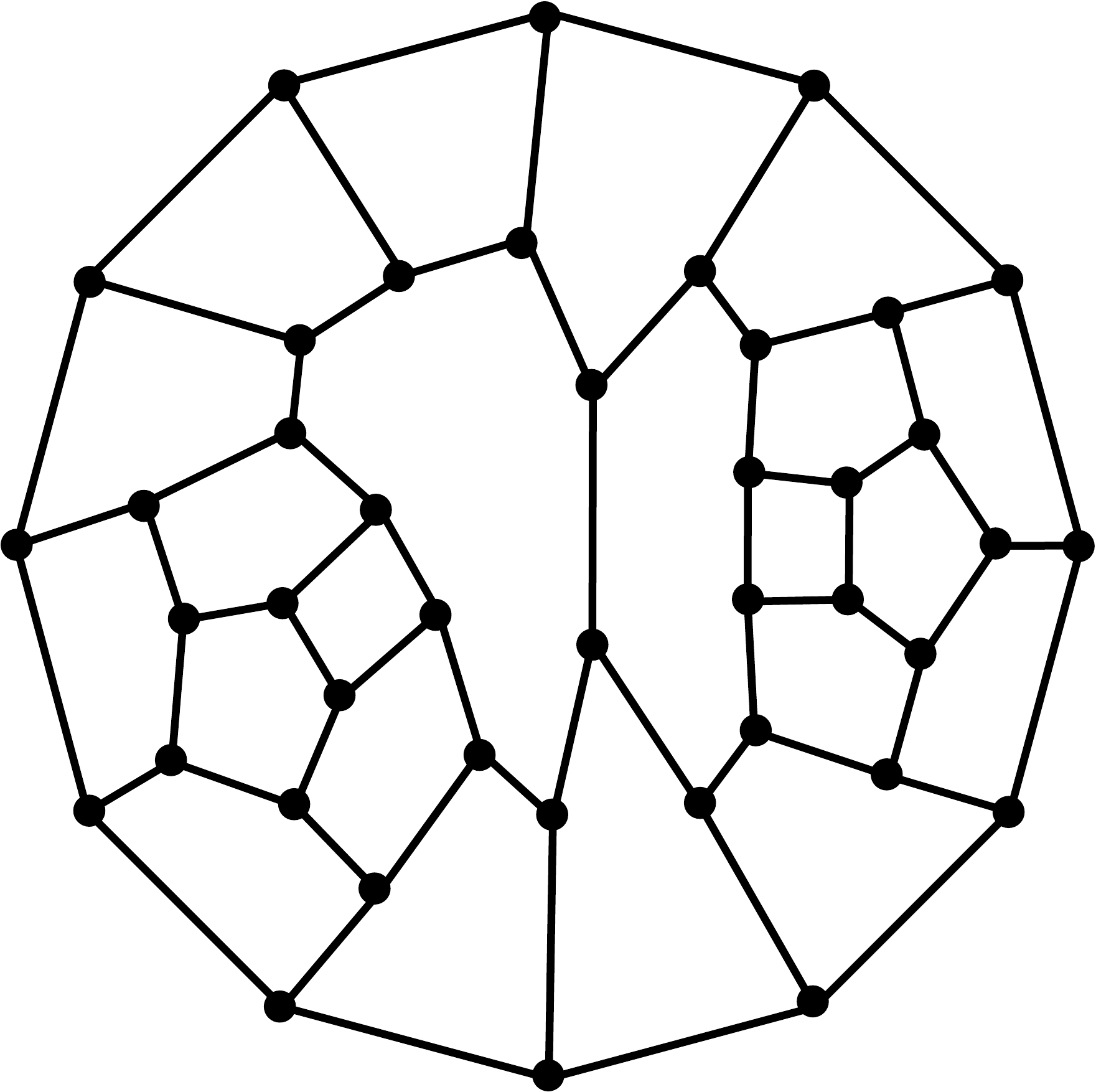}} \qquad
    \subfloat[]{\label{fig:planar_42v_4HC_subdiv}\includegraphics[width=0.28\textwidth]{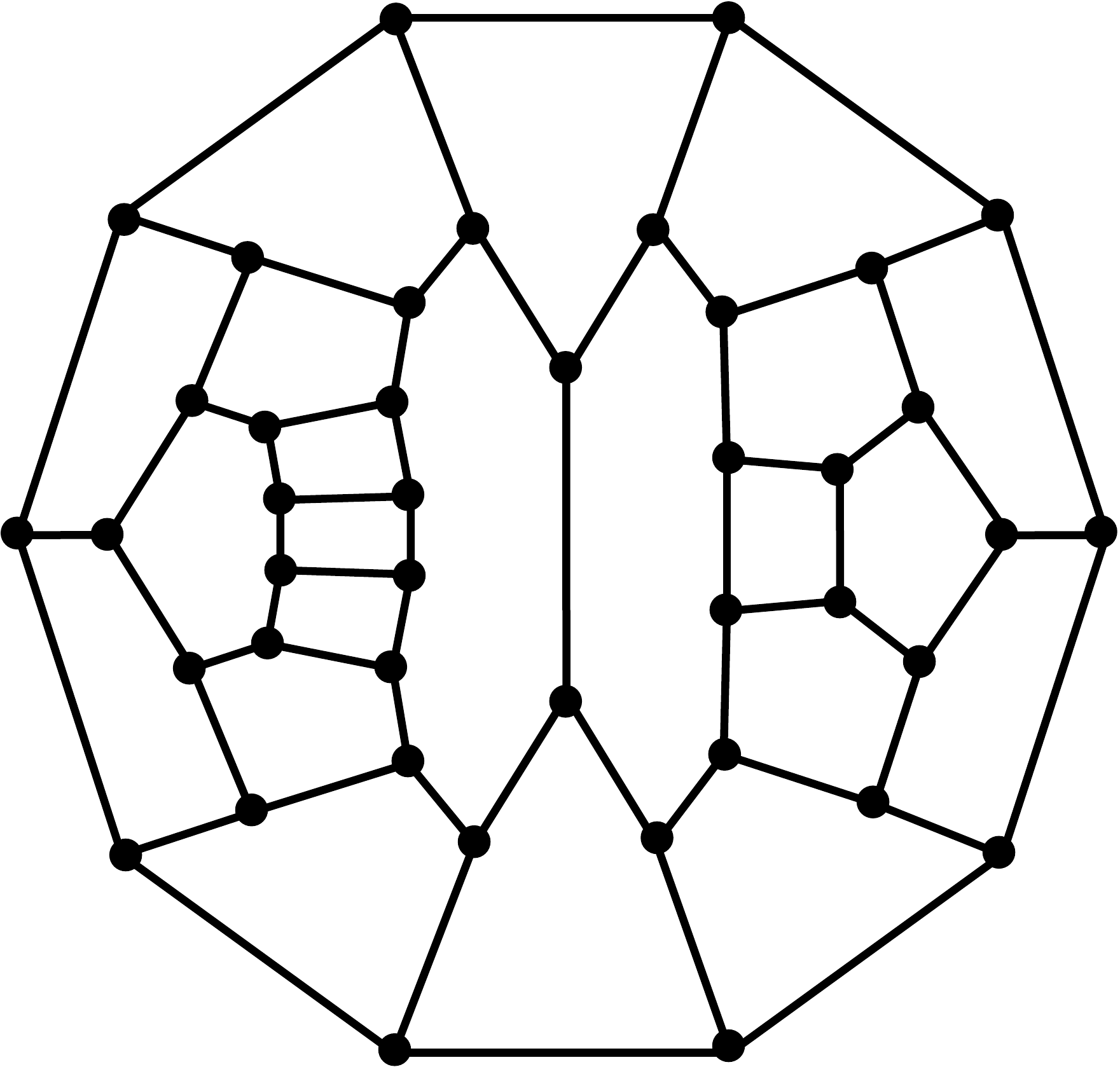}} \qquad
    \subfloat[]{\includegraphics[width=0.28\textwidth]{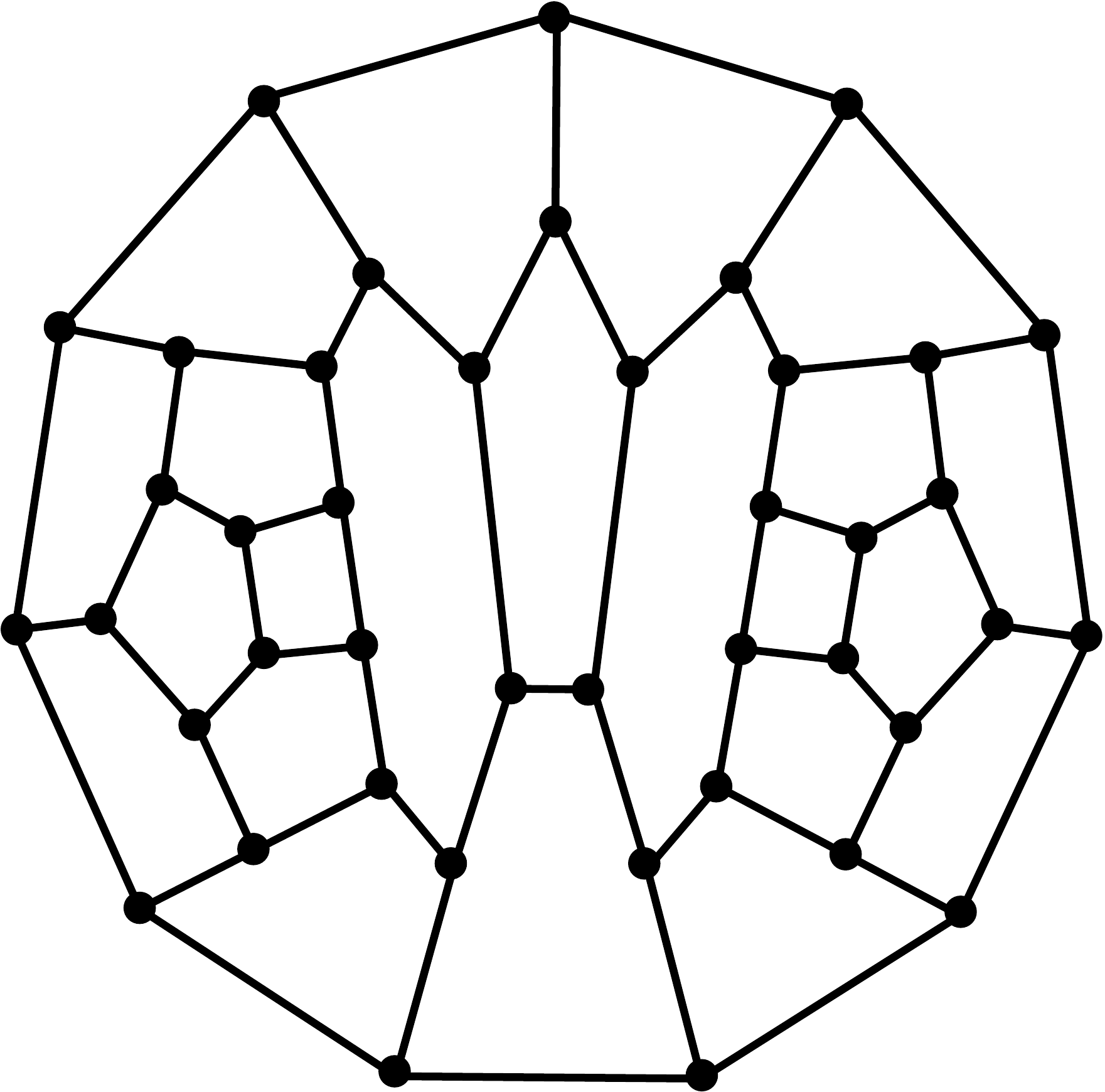}} \qquad
    \subfloat[]{\includegraphics[width=0.28\textwidth]{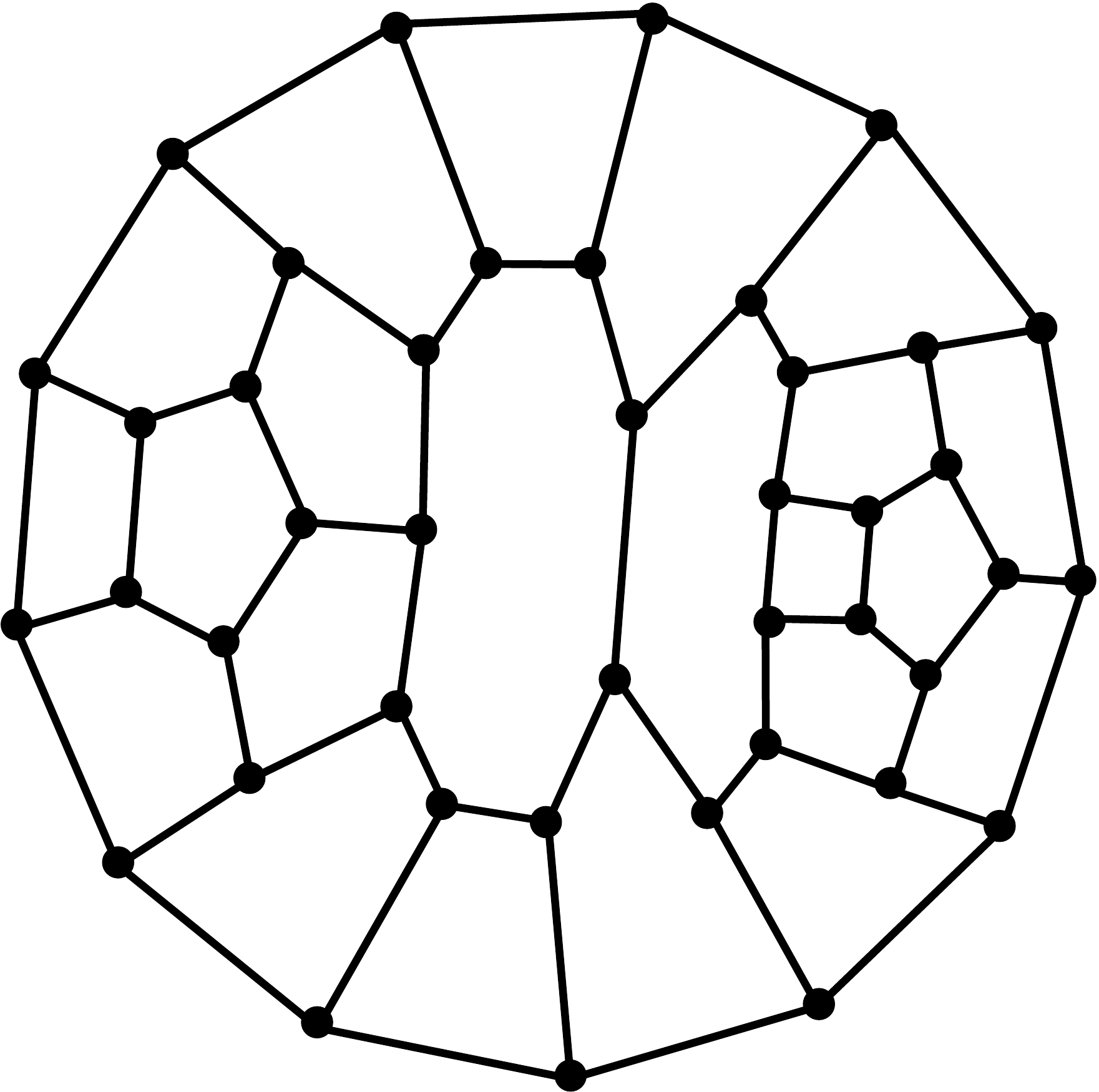}}
    \caption{The five planar cyclically $4$-edge-connected cubic graphs with exactly four hamiltonian cycles on $42$ vertices from Observation~\ref{obs:pl_4HC}.}
	\label{fig:planar_42v_4HC}
\end{figure}

\end{document}